\documentclass[12pt]{amsart}
\usepackage{a4wide}
\usepackage[utf8]{inputenc}
\usepackage{amsmath}
\usepackage{amsthm}
\usepackage{amssymb}
\usepackage{amsfonts}
\usepackage{amsopn}
\usepackage{graphicx}
\usepackage{enumerate}
\usepackage{color}
\usepackage{mathtools}
\usepackage{mathrsfs}
\usepackage{bbm}
\usepackage{yhmath}
\usepackage{cite}
\usepackage[colorlinks,linkcolor=blue,anchorcolor=blue,citecolor=blue]{hyperref}
\usepackage{engord}

\newtheorem{theorem}{Theorem}[section]
\newtheorem{proposition}[theorem]{Proposition}
\newtheorem{lemma}[theorem]{Lemma}

\newtheorem{corollary}[theorem]{Corollary}

\theoremstyle{definition}
\newtheorem{example}[theorem]{Example}

\numberwithin{equation}{section}

\newcommand{\mcal}{\mathcal}

\newcommand{\set}[1]{\left\{#1 \right\}}

\newcommand{\R}{\mathbb{R}}

\newcommand{\Z}{\mathbb{Z}}
\newcommand{\N}{\mathbb{N}}

\newcommand{\f}{\infty}
\newcommand{\la}{\langle}
\newcommand{\ra}{\rangle}

\newcommand{\sse}{\subseteq}

\newcommand{\ba}{\boldsymbol{\alpha}}
\newcommand{\dimh}{\dim_{\rm H}}

\newcommand{\dimp}{\dim_{\rm P}}

\newcommand{\bp}{\mathbf{p}}
\newcommand{\bu}{\boldsymbol{\alpha}}
\newcommand{\bv}{\boldsymbol{\beta}}

\newcommand{\bn}{\mathbf{n}}

\makeatletter
\newcommand{\rmnum}[1]{\romannumeral #1}

\makeatother

\title[Existence and Spectrality of random measures]{Existence and Spectrality of random measures generated by infinite convolutions}

\author[H. Liu]{Hongyi Liu}
\address[H. Liu]{School of Mathematical Sciences,  Key Laboratory of MEA(Ministry of Education) \& Shanghai Key Laboratory of PMMP,  East China Normal University, Shanghai 200241, China}
\email{51265500101@stu.ecnu.edu.cn}

\author[J. J. Miao]{Jun Jie Miao}
\address[J. J. Miao]{School of Mathematical Sciences,  Key Laboratory of MEA(Ministry of Education) \& Shanghai Key Laboratory of PMMP,  East China Normal University, Shanghai 200241, China }
\email{jjmiao@math.ecnu.edu.cn}

\author[H. Zhao]{Hongbo Zhao}
\address[H. Zhao]{School of Mathematical Sciences,  Key Laboratory of MEA(Ministry of Education) \& Shanghai Key Laboratory of PMMP,  East China Normal University, Shanghai 200241, China}
\email{2504245357@qq.com}

\date{\today}

\begin{document}
	
	\maketitle
	\begin{abstract}
		In this paper, we construct a class of random measures $\mu^{\mathbf{n}}$ by infinite convolutions. Given infinitely many admissible pairs $\{(N_{k}, B_{k})\}_{k=1}^{\infty}$ and a  positive integral sequence  $\bn=\{n_{k}\}_{k=1}^{\infty}$, for every $\boldsymbol{\omega}\in \mathbb{N}^{\mathbb{N}}$, we write $\mu^{\mathbf{n}}(\boldsymbol{\omega}) =  \delta_{N_{\omega_{1}}^{-n_{1}}B_{\omega_{1}}} * \delta_{N_{\omega_{1}}^{-n_{1}}N_{\omega_{2}}^{-n_{2}}B_{\omega_{2}}} * \cdots$. If $n_{k}=1$ for $k\geq 1$, write $\mu(\boldsymbol{\omega})=\mu^{\mathbf{n}}(\boldsymbol{\omega})$.
		First, we show that the mapping $\mu^{\mathbf{n}}: (\boldsymbol{\omega}, B) \mapsto \mu^{\mathbf{n}}(\boldsymbol{\omega})(B)$ is a random measure if the family of Borel probability measures $\{\mu(\boldsymbol{\omega}) : \boldsymbol{\omega} \in \mathbb{N}^{\mathbb{N}}\}$  is tight. Then, for every Bernoulli measure $\mathbb{P}$ on $\mathbb{N}^{\mathbb{N}}$, the random measure $\mu^{\mathbf{n}}$ is also a spectral measure $\mathbb{P}$-a.e.. If the positive integral sequence $\bn$ is unbounded,  the random measure $\mu^{\mathbf{n}}$ is  a spectral measure  regardless of the measures  on the sequence space $\mathbb{N}^{\mathbb{N}}$. Moreover,  we provide some sufficient conditions for the existence of the random measure $\mu^\bn$. Finally, we verify that random measures have the intermediate-value property.
		

	\end{abstract}

	\section{Introduction}\label{section_introduction}
	\subsection{Spectral measures and fractals}
	A Borel probability measure $\mu$ on $\R^d$ is called a \emph{spectral measure} if there exists a set $\Lambda \subset \R^d$ such that the family of exponential functions
	$$E(\Lambda)= \big\{ e_\lambda(x) = e^{2\pi i \lambda \cdot x} : \lambda \in \Lambda \big\}$$
	forms an orthonormal basis for $L^2(\mu)$, where the set $\Lambda$ is called a \emph{spectrum} of $\mu$.
	In  Fourier analysis, the Lebesgue measure on the hypercube $[0,1]^d$ is a spectral measure with a spectrum $\Z^d$, and its support  exhibits a strong geometric structure. In 1974, Fuglede proposed the following well-known spectral set conjecture, see ~\cite{Fuglede-1974}.
	\begin{quote}
		\textbf{Conjecture}: \emph{Let $\Gamma \subset \R^d$ be a measurable set with positive finite Lebesgue measure. Then there exists a set $\Lambda \subset \R^d$ such that $\{ e_{\lambda}(x) = e^{2\pi i \lambda \cdot x}: \lambda \in \Lambda\}$ forms an orthogonal basis for $L^2(\Gamma)$,
			if and only if $\Gamma$ tiles $\R^d$ by translations.}
	\end{quote}
In 2004,  Tao \cite{Tao-2004} gave the first counterexample in $\R^d$ for $d \ge 5$.
	From then on, more counterexamples were  constructed in $\R^d$ for $d=2, 3$, see~\cite{Farkas-Revesz-2006, Farkas-Matolcsi-Mora-2006,Kolountzakis-Matolcsi-2006a,Kolountzakis-Matolcsi-2006b,Matolcsi-2005}.
	Recently, Nev and Matolcsi \cite{Lev-Matolcsi-2019} showed that the spectral set conjecture holds in all dimensions for convex domains.
	
	Fractal  measures are important research objects in fractal geometry which are frequently singular continuous  with respect to Lebesgue measures, and we refer readers to~\cite{Bk_KJF2} for the background reading. Such measures also have  many surprising phenomena in spectral theory.
	In 1998, Jorgensen and Pedersen \cite{Jorgensen-Pedersen-1998} discovered that the self-similar measure $\mu_{4,\{0,2\}}$ given by
	\begin{equation}\label{ssm14}
		\mu(\;\cdot\;) = \frac{1}{2} \mu(4 \;\cdot\;) + \frac{1}{2} \mu( 4 \;\cdot\; -2)\end{equation}
	is a spectral measure with a spectrum
	\begin{equation}\label{eq:Lambda-4-0-1}
		\Lambda = \bigcup_{k=1}^\f \big\{ \ell_1 + 4 \ell_2 + \cdots + 4^{k-1} \ell_k: \ell_1, \ell_2, \ldots, \ell_k \in \{0,1\} \big\},
	\end{equation}
see \cite{An-Fu-Lai-2019,An-He-He-2019,An-He-Lau-2015,Dai-He-Lau-2014, Fu-He-Wen-2018, He-Tang-Wu-2019, Li-2009, Liu-Dong-Li-2017}  for the study of various  fractal spectral measures.
	
In  \cite{Dutkay-Jorgensen-2012},  Dutkay and Jorgensen revealed an interesting fact  that besides the set $\Lambda$ defined in \eqref{eq:Lambda-4-0-1}, the sets $5\Lambda, 7\Lambda, 11\Lambda, 13\Lambda, 17\Lambda, \ldots$ are all spectra of $\mu_{4,\{0,2\}}$. This scaling property was first found by Laba and Wang for self-similar measure in \cite{Laba-Wang-2002},  and it has been extended to other fractal spectral measures. See \cite{Dutkay-Han-Sun-2009, Fu-He-2017,He-Tang-Wu-2019} for details.
	Therefore  the convergence of the Fourier series of functions $$\sum_{\lambda \in \Lambda}~ \la f, e_\lambda \ra_{L^2(\mu)}\;e_\lambda(x)$$
	may be very different for distinct spectra of singularly continuous spectral measures.
	For the fractal measure $\mu_{4,\{0,2\}}$ given by \eqref{ssm14}, Strichartz \cite{Strichartz-2006} proved that the mock Fourier series of continuous functions converges uniformly with respect to the spectrum $\Lambda$ given by \eqref{eq:Lambda-4-0-1}. However,  Dutkey, Han and Sun \cite{Dutkay-Han-Sun-2014} showed that there exists a continuous function such that its mock Fourier series is divergent at $0$ with respect to the spectrum $17\Lambda$.
	In addition, the spectra of some fractal measure may be very rich, Li and Wu showed in \cite{Li-Wu-2024} that for some spectral Moran  measures, the Beurling dimension of spectra has the intermediate value property.  These interesting results indicate that singular fractal spectral measures may have more complex geometric structure and more intricate analysis properties compared to absolutely continuous spectral measures   with respect to Lebesgue measures.
	
	\subsection{Infinite convolutions}
	A key strategy to study the spectra theory of fractal measures such as self-affine measure and Moran measure, is  by using infinite convolutions. Note that both self-affine measures and Moran measures may be regarded as generalizations of self-similar measures.
	
	Let $\delta_a$ be the Dirac measure concentrated on the point $a$. Given a finite subset $A \subset \R$.  We write
	$$
	\delta_A = \frac{1}{\# A} \sum_{a\in A} \delta_a,
	$$
	where $\#$ denotes the cardinality of a set.
	Let $\{ A_k\}_{k=1}^\f$ be a sequence of finite subsets of $\R$ such that $\# A_k \ge 2$ for every $k \ge 1$.
	For  each integer $k \geq 1$, we define
	\begin{equation}\label{discrete-convolution}
		\nu_k =\delta_{A_1}*\delta_{A_2} * \cdots *\delta_{A_k},
	\end{equation}
	where $*$ denotes the convolution between measures.
	If the sequence of convolutions $\{\nu_k\}_{k=1}^\f$ converges weakly to a Borel probability measure $\nu$, then we call $\nu$ the \emph{infinite convolution} of $\{{A_k}\}_{k=1}^\infty$, denoted by
	\begin{equation}\label{infinite-convolution}
		\nu =\delta_{A_1}*\delta_{A_2} * \cdots *\delta_{A_k} *\cdots.
	\end{equation}

	It is clear that the uniformly distributed self-affine measures and Moran measures may be regarded as special cases of infinite convolutions. Given a sequence  $\{(N_k,B_k)\}_{k=1}^\infty$ where $N_k\ge2$ and $B_k\subset\R$ is finite for all $k\in\N$. Then we write
	\begin{equation}\label{def_ick}
		\mu_{k} = \delta_{N_1^{-1} B_1} * \delta_{(N_1 N_2)^{-1} B_2} * \cdots * \delta_{(N_1N_2 \cdots N_k)^{-1} B_k}.
	\end{equation}
	If the sequence $\left\{\mu_k\right\}_{k=1}^\infty$ converges weakly to a Borel probability measure $\mu$, then we call $\mu$ the \emph{infinite convolution} of $\left\{(N_k,B_k)\right\}_{k=1}^\infty $, denoted by
	\begin{equation}\label{def_ic}
		\mu =\delta_{{N_{1}}^{-1}B_1}\ast\delta_{(N_1N_2)^{-1}B_2}\ast\dots\ast\delta_{(N_1N_{2}\cdots N_k)^{-1}B_k} *\cdots.
	\end{equation}			

Admissible pairs are the key to study the spectrality of infinite convolutions. Given an integer $N\ge2$ and a finite subset  $B\sse \Z$ with $\# B\ge 2$.  If there exists $L\sse \Z^d$ such that the matrix
	$$
	\left[ \frac{1}{\sqrt{\# B}} e^{-2 \pi i  \frac{b\cdot l}{N}}  \right]_{b \in B, l \in L} $$
	is unitary, we call $(N, B)$ an {\it admissible pair} in $\R$ and call $(N,B,L)$ a {\it Hadamard triple} in $\R$, see~\cite{Dutkay-Haussermann-Lai-2019} for details. 	Spectrality of infinite convolutions generated by a sequence of admissible pairs was first studied by Strichartz, where he constructed the spectrum under a specific uniform separation condition in \cite{Strichartz-2000}. But in general, the condition is difficult to verify.
	If the infinite convolution defined in \eqref{def_ic} exists, then it  is of pure type,  see \cite[Theorem~35]{Jessen-Wintner-1935} for detail.  If the elements of $\{(N_k,B_k)\}_{k=1}^\infty$ are identical, that is, $(N_k,B_k)=(N,B)$ for all integers $k>0$, the corresponding infinite convolution is a self-similar measure, denoted by $\mu_{N,B}$. {\L}aba and Wang \cite{Laba-Wang-2002} showed that if $(N,B)$ is an admissible pair, then the self-similar measure $\mu_{N,B}$ is a spectral measure, and Dutkay, Haussermann and Lai \cite{Dutkay-Haussermann-Lai-2019} generalized it to self-affine measures in higher dimensions.

	The admissible pairs are a crucial requirement in  spectral theory  of infinite convolutions since they provide an  infinite and mutually orthogonal set of exponential functions. Therefore, the difficulty to prove spectrality is to show the completeness of the orthogonal set for a given infinite convolution.  However, the admissible pairs are not enough to guarantee that the corresponding infinite convolution is a spectral measure (see Example 4.3 in \cite{An-He-He-2019}), even if the sequence of admissible pairs is chosen from a finite set of admissible pairs (see Example 1.8 in \cite{Dutkay-Lai-2017}). Nevertheless, it is widely believed that negative examples are very rare.
	An, Fu and Lai \cite{An-Fu-Lai-2019} introduced the concept of equi-positivity, and used the integral periodic zero set to define an admissible family, both of which have been extensively manipulated in analyzing the spectrality of infinite convolutions \cite{Li-Miao-Wang-2021,Li-Miao-Wang-2022,Liu-Lu-Zhou-2023b}.

	\subsection{Random measures and Main conclusions}
	In this paper, we apply infinite convolutions to construct a class of random measures, and we study the spectrality of such random measures. First, we recall the definition of random measures. Let $(\Omega, \mathcal{F})$ and $(E,\mathcal{E})$ be measurable spaces. A mapping $M: \Omega \times \mathcal{E} \to [0,+\infty]$ is called a \emph{random measure} on $(E,\mathcal{E})$ if $(\rmnum{1})$  the mapping $\boldsymbol{\omega} \mapsto M(\boldsymbol{\omega}, A)$ is $\mathcal{F}$-measurable for every $A\in\mathcal{E}$; and $(\rmnum{2})$ $A \mapsto M(\boldsymbol{\omega}, A)$ is a measure on $(E,\mathcal{E})$ for all $\boldsymbol{\omega}\in\Omega$. See \cite{Cinlar11} for details.

	For each $k=1,2,\cdots$, we write
	\begin{equation}\label{defsigmak}
		\Omega^k =\mathbb{N}^k=\{\alpha_{1}\alpha_{2}\cdots \alpha_{k}: \alpha_j\in\mathbb{N},\ j=1,2,\ldots, k\}
	\end{equation}
	for the set of sequences of length $k$, with $
	\Omega^{0}=\{\emptyset \}$ containing only the empty sequence $\emptyset$, and write
	\begin{equation}\label{defsigma*}
		\Omega^{*}=\bigcup _{k=0}^{\infty }\Omega^{k}
	\end{equation}
	for the set of all finite sequences. We write
	$$
	\Omega=\mathbb{N}^{\mathbb{N}}= \{\alpha_{1}\alpha_{2}\cdots \alpha_{j}\cdots :  \alpha_j \in\mathbb{N}\}
	$$
	for the corresponding set of infinite sequences.
	For $\bu=\alpha_1\cdots \alpha_k\in\Omega^k$, we write $\bu^-=\alpha_1\cdots \alpha_{k-1}$ and  write $|\bu|=k$ for the length of $\bu$. For each $\bu = \alpha_1 \alpha_2\cdots \alpha_k\in
	\Omega^{*}$, and $\bv = \beta_1 \beta_2\cdots \in \Omega$,   we say $\bu$ is a \textit{curtailment or prefix} of $\bv$, denoted by  $\bu \preceq \bv$, if $\bu =\bv|k= \beta_1\cdots \beta_k $. We call the set $[\bu] =
	\{\bv\in\Omega: \bu \preceq\bv\}$ the \textit{cylinder} of $\bu$. If $\bu=\emptyset$, its cylinder is $[\bu]=\Omega$.  We term a subset $A$ of $\Omega^*$ a	\textit{cut set} if $\Omega\subset\bigcup_{\bu\in A}[\bu]$, where $[\bu]\bigcap [\bv]=\emptyset$ for  all $\bu\neq\bv\in A$.
If	$\bu,\bv\in \Omega$, then we denote by $\bu\wedge\bv \in \Omega^*$ the
	maximal common initial subsequence of both $\bu$ and $\bv$.

We topologise $\Omega$  using the metric $d(\bu,\bv) = 2^{-|\bu \wedge \bv|}$ for all $\bu, \bv \in \Omega$ to make $\Omega$  into a complete metric space, see \cite{BBT08} for details.  Let $\mathcal{F}$ be the Borel $\sigma$-algebra  on $\Omega$. Then $(\Omega, \mathcal{F})$ is a measurable space. We write $\mathcal{P}(\Omega)$ for the set of all Borel probability measures on $\Omega$.

	Given a probability vector $\mathbf{p}=(p_{1}, p_{2}, \dots)$, i.e., $\sum_{j=1}^{\infty} p_j=1$ where $p_j\geq 0$ for all $j >0$,  we define a probability measure $\mathbb{P}$ on $\Omega$ by setting
	\begin{equation} \label{def_BPM}
		\mathbb{P}([\bu]) =p_{\bu}\equiv p_{\alpha_{1}}p_{\alpha_{2}}\cdots p_{\alpha_{k}}  \quad (\bu=\alpha_1\cdots \alpha_k)
	\end{equation}
	for each cylinder $[\bu]$ and
	extending to general subsets of $\Omega$ in the usual way. The probability measure $\mathbb{P}$ is called the {\it Bernoulli measure} associated with the probability vector $\mathbf{p}$. A probability vector $\mathbf{p}=(p_{1}, p_{2}, \dots)$ is called {\it positive} if $p_{j}>0$ for all $j\geq 1$.

We define random measures on $(\mathbb{R},\mathcal{B}(\mathbb{R}))$ with respect to $(\Omega, \mathcal{F})$ by  infinite convolutions. Given a sequence $\{(N_k,B_k)\}_{k=1}^\f$ where  $N_k\ge2$ and $B_k\subset \R $ is finite for all $k>0$. In this paper, we always write  $\bn=\{ n_{k}\}_{k=1}^\f$ for a sequence of positive integers. For each integer $k>0$, we write that
	\begin{equation} \label{def_mubk}
		\mu^{\bn}_{k}(\boldsymbol{\omega}) =\delta_{N_{\omega_{1}}^{-n_{1}}B_{\omega_{1}}} * \delta_{N_{\omega_{1}}^{-n_{1}}N_{\omega_{2}}^{-n_{2}}B_{\omega_{2}}} *  \dots * \delta_{N_{\omega_{1}}^{-n_{1}}N_{\omega_{2}}^{-n_{2}}\cdots N_{\omega_{k}}^{-n_{k}}B_{\omega_{k}}},
	\end{equation}
	for all $\boldsymbol{\omega}\in \Omega$, and this defines a mapping $	\mu_{k}^{\mathbf{n}} : \Omega\times \mathcal{B}(\mathbb{R}) \to [0,+\f]$  by $	\mu_{k}^{\mathbf{n}}(\boldsymbol{\omega}, B) = \mu_{k}^{\mathbf{n}}(\boldsymbol{\omega})(B)$ for all  $\boldsymbol{\omega} \in \Omega$ and $B\in\mathcal{B}(\mathbb{R})$. The following conclusion shows that    $\mu^{\bn}_{k}$ is a random measure.
\begin{theorem} \label{thm_measurable'}
		Given a sequence $\{( N_k,B_k)\}_{k=1}^\infty $ where  $N_k\ge2$ and $B_k\subset \R $ is finite for all $k>0$. Then for every  sequence $\mathbf{n}$ of positive integers,  the mapping $\mu_{k}^{\mathbf{n}}$ given by \eqref{def_mubk} is  a random measure for all $k>0$.
	\end{theorem}
	
It is more important to explore the limit behavior of  $\mu_{k}^{\mathbf{n}}$ which is the main object studied in this paper. Suppose that for every $\boldsymbol{\omega}\in \Omega$, the sequence $\{\mu_k^{\mathbf{n}}(\boldsymbol{\omega})\}_{k=1}^\infty$ converges weakly to $\mu^{\mathbf{n}}(\boldsymbol{\omega})$, written as
	\begin{equation} \label{def_mubn}
		\mu^{\mathbf{n}}(\boldsymbol{\omega}) = \delta_{N_{\omega_{1}}^{-n_{1}}B_{\omega_{1}}} * \delta_{N_{\omega_{1}}^{-n_{1}}N_{\omega_{2}}^{-n_{2}}B_{\omega_{2}}} * \cdots .
	\end{equation}
	We define a mapping $\mu^{\mathbf{n}}: \Omega\times \mathcal{B}(\mathbb{R}) \to [0,+\f]$ by
	\begin{equation} \label{def_random_measure}
		\mu^{\mathbf{n}}(\boldsymbol{\omega}, B)= \mu^{\mathbf{n}}(\boldsymbol{\omega})(B),
	\end{equation}
	for all $\boldsymbol{\omega} \in \Omega$ and all   $B \in \mathcal{B}(\mathbb{R})$.
	For simplicity,  if $\bn=\{ 1 \}_{k=1}^\f$, we write
	\begin{equation} \label{def_corresponding_P_of_random_measure}
		\mu(\boldsymbol{\omega}) =		\mu^\bn(\boldsymbol{\omega}) = \delta_{N_{\omega_{1}}B_{\omega_{1}}} * \delta_{N_{\omega_{1}}N_{\omega_{2}}B_{\omega_{2}}} * \cdots,
	\end{equation}
	for all $\boldsymbol{\omega}\in\Omega$.
	
	We assume that $\mu(\boldsymbol{\omega})$ exists for every $\boldsymbol{\omega}\in\Omega$ and write
	\begin{equation} \label{def_setPhi}
		\Phi = \{ \mu(\boldsymbol{\omega}): \boldsymbol{\omega} \in \Omega \}.
	\end{equation}
	It turns out that the mapping $\mu^{\mathbf{n}}$ is a random measure under the tightness assumption. See Section \ref{section_preliminaries} for the definition of tightness.
\begin{theorem} \label{thm_measurable}
		Given a sequence $\{( N_k,B_k)\}_{k=1}^\infty $ where $N_k\ge2$ and $B_k\subseteq [0,+\f)$ is finite for all $k>0$. Suppose that $\Phi$ given by \eqref{def_setPhi} is tight. Then for every  sequence $\mathbf{n}$ of positive integers,  the mapping $\mu^{\mathbf{n}}$ given by \eqref{def_random_measure} is a random measure.
	\end{theorem}

	Note that in the assumption that  $\Phi$ is tight, we assume  $\mu(\boldsymbol{\omega})$ exists first for every $\boldsymbol{\omega} \in \Omega$. By Proposition~\ref{prop_tightness_guarantees_existence}, the existence of  $\mu(\boldsymbol{\omega})$ implies that   $\mu^{\mathbf{n}}(\boldsymbol{\omega})$ exists for all positive integral  sequence $\mathbf{n}$. Hence the mapping  $\mu^{\mathbf{n}}$ in Theorem \ref{thm_measurable} is well-defined. The existence and tightness of $\Phi$  are  explored in Section \ref{sec_ex}, and  some sufficient conditions are provided later.

	Our main purpose is to study the spectrality of random measures.
Let  $M$ be a random measure on  $(E,\mathcal{E})$ with respect to  $(\Omega, \mathcal{F})$. We say $M$ is a \emph{spectral random measure} if $M(\boldsymbol{\omega})$ is a spectral measure for all $\boldsymbol{\omega}\in\Omega$. Moreover, given a Borel probability measure $\mathbb{P} \in \mathcal{P}(\Omega)$, we say $M$ is \emph{a spectral random for $\mathbb{P}$-a.e. $\omega\in \Omega$  (or  the random measure $M$ is spectral $\mathbb{P}$-a.e.)} if $M(\boldsymbol{\omega})$ is a spectral measure for $\mathbb{P}$-a.e. $\boldsymbol{\omega}\in\Omega$.
	
	If the sequence $\bn= \{n_{k}\}_{k=1}^\f $ is unbounded, we obtain that $\mu^{\bn}$ is a spectral random measure with assumption of admissible pairs.
	\begin{theorem}\label{thm_ub}
		Given a sequence $\{( N_k,B_k)\}_{k=1}^\infty $ of admissible pairs where $B_k \subseteq \mathbb{N}$ for all $k>0$, suppose that $\Phi$ given by \eqref{def_setPhi} is tight. Then for every unbounded sequence $\mathbf{n}$ of positive integers,  the mapping $\mu^{\bn}$ given by \eqref{def_random_measure} is a  spectral random measure.
	\end{theorem}
	
	If we remove the unboundedness of $\bn$, we conclude that  $\mu^{\mathbf{n}}$ is  a spectral random measure  almost surly.
	\begin{theorem}\label{thm_TC}
		Given a sequence $\{( N_k,B_k)\}_{k=1}^\infty $ of admissible pairs where $0\in B_k \subseteq \mathbb{N}$ for all $k>0$,   suppose that $\Phi$ given by \eqref{def_setPhi} is tight. Then for every  Bernoulli probability $\mathbb{P}$ on $\Omega$  and  every sequence $\mathbf{n}$ of positive integers, the mapping $\mu^{\bn}$ given by \eqref{def_random_measure} is a  spectral random measure for  $\mathbb{P}$-a.e. $\boldsymbol{\omega}\in\Omega$.
	\end{theorem}

	All these conclusions depend on the existence of $\mu(\boldsymbol{\omega})$ and the tightness of $\Phi$, and generally, both the existence of $\mu(\boldsymbol{\omega})$ and the tightness of $\Phi$ are very difficult to verify. Next, we provide some sufficient conditions for them.
	
	Given a sequence $\{( N_k,B_k)\}_{k=1}^\infty $, we say $\{( N_k,B_k)\}_{k=1}^\infty $ satisfies {\it remainder bounded condition (RBC)} if
		\[\sum_{k=1}^{\infty} \frac{\# B_{k,2}}{\# B_k} < \infty,\]
		where $B_{k,1}=B_k \cap \{0,1,\cdots,N_k-1\}$ and $B_{k,2}=B_k \backslash B_{k,1}.$ The following conclusion provides a sufficient condition for the existence of $\mu(\boldsymbol{\omega})$ and the tightness of $\Phi$.
	\begin{theorem}\label{thm_tight}			
		Let $\{( N_k,B_k)\}_{k=1}^\infty $ be a sequence satisfying {\it RBC} where $0\in B_k \subseteq \mathbb{N}$ and $\# B_k\le N_k$ for all $k>0$.  Suppose that
		$$
		\sup_{k\ge 1}\frac{\log\max B_k}{\log N_k}<\f.
		$$ 			
		Then the infinite convolution $\mu(\boldsymbol{\omega})$ given by \eqref{def_corresponding_P_of_random_measure} exists for every $\boldsymbol{\omega} \in \Omega$. Moreover $\Phi$ given by \eqref{def_setPhi}  is tight.
	\end{theorem}
Note that  $\# B_k\le N_k$ assumed in Theorem~\ref{thm_tight} is because we do not use  admissible pairs in the sequence $\{( N_k,B_k)\}_{k=1}^\infty $,	and it is automatically satisfied for admissible pairs.
	
	The following  simple condition may be more useful to obtain spectral random measures in practice.
	\begin{corollary}\label{thm_CC}
		Given a sequence of admissible pairs $\{( N_{k}, B_{k})\}_{k=1}^{\infty}$ where $0\in B_k \subseteq \mathbb{N}$ for all $k>0$ satisfying
		\begin{equation}\label{cdn_NBIF}
			\sup_{k\geq1} \{N_{k}^{-1}b:b\in B_{k}\} < \infty.
		\end{equation}
Let $\mu^{\bn}$ be the mapping given by \eqref{def_random_measure}.  		Then for every sequence $\bn$ of positive integers,
		
		$(\rmnum{1})$ the mapping $\mu^{\bn}$ is a random measure;
		
		$(\rmnum{2})$ for every Bernoulli measure $\mathbb{P}$ on $\Omega$, $\mu^{\bn}$ is spectral $\mathbb{P}$-a.e.

		$(\rmnum{3})$ if $\mathbf{n}$ is unbounded,  $\mu^{\bn}$ is a spectral random measure.
	\end{corollary}
Note that the assumption \eqref{cdn_NBIF} actually implies that for every $\boldsymbol{\omega} \in \Omega$, the support of  the realization $\mu^{\mathbf{n}}(\boldsymbol{\omega})$  is contained in a common compact set, and this  means $\Phi$ is tight.   In Example \ref{ex_tnc},  we construct a $\Phi$ which is tight but with   no common compact support.

	Suppose that  the sequence of admissible pairs $\{(N_k,B_k)\}_{k=1}^\f$ only consists of finitely many admissible pairs, that is $\{(N_k,B_k)\}_{k=1}^M.$ 	
	We write $\Omega_M=\{1,2,\ldots, M\}^{\mathbb{N}}$ and define the mapping $\mu^\bn$  the same as  \eqref{def_random_measure}  using finitely many admissible pairs $\{(N_k,B_k)\}_{k=1}^M.$ Note that $\Omega_M$ is a compact subset of $\Omega$. 	In this special case, the assumption \eqref{cdn_NBIF} is automatically satisfied, and we have the following conclusions.
	
	\begin{corollary}\label{cor_fad}
		Given  finite admissible pairs $\{( N_{k}, B_{k})\}_{k=1}^M$ where $0\in B_k \subseteq \mathbb{N}$ for all $k>0$. Let $\mu^{\bn}$ be the mapping given by \eqref{def_random_measure}. Then for every sequence $\mathbf{n}$ of positive integers,
		
		$(\rmnum{1})$ the mapping  $\mu^{\bn}$ is a random measure;
		
		$(\rmnum{2})$ for every Bernoulli measure $\mathbb{P}$ on $\Omega_M$,  $\mu^{\bn}$ is spectral $\mathbb{P}$-a.e.
		
		$(\rmnum{3})$ if $\mathbf{n}$ is unbounded,  $\mu^{\bn}$ is a spectral random measure.
	\end{corollary}

Finally, it is worth  to point out that spectral random measures are not rare, and the following conclusion reveals that  spectral random measures are actually plenty from dimension point of view, that is, they have the intermediate-value property in dimensions. Moreover, even for a given spectral random measure, it may also have very rich geometric structures, see Corollary \ref{cor_dimensions}.	We refer readers to \cite{Bk_KJF2} for details on dimension theory.
	\begin{theorem} \label{thm_imp}
		For every $s \in (0,1]$, there exist a spectral random measure $\mu$  and a Bernoulli probability measure $\mathbb{P} $ on $\Omega$ such that
		$$
		\dimh\mu =\dimp\mu = s,
		$$
$\mathbb{P}$-almost surely, where $\dimh$ and $\dimp$ denote the Hausdorff dimension and packing dimension, respectively.
	\end{theorem}

	\section{Preliminaries}\label{section_preliminaries}	
	\subsection{Fourier Transform and Weak Convergence}
	We write $\mathcal{P}(\mathbb{R})$ for the set of all Borel probability measures on $\mathbb{R}$ and $C_{b}(\mathbb{R})$  for the set of all bounded continuous functions on $\mathbb{R}$. Given $\mu, \mu_{1}, \mu_{2}, \dots \in \mathcal{P}(\mathbb{R})$, 	 we say that $\mu_{k}$ converges weakly to $\mu$ if
	\begin{equation*}
		\lim\limits_{k\to\infty} \int_{\mathbb{R}} f(x) \mathrm{d}\mu_{k}(x) = \int_{\mathbb{R}} f(x) \mathrm{d}\mu(x),
	\end{equation*}
	for every $f\in C_{b}(\mathbb{R})$. Given a subset $\Psi \subseteq \mathcal{P}(\mathbb{R})$, we say that $\Psi$ is \emph{tight} (sometimes called {\it uniformly tight}) if for each $\epsilon > 0$, there exists a compact subset $K \subseteq \mathbb{R}$ such that
	\begin{equation*}
		\inf_{\mu\in\Psi} \mu(K) > 1 - \epsilon,
	\end{equation*}
	see \cite{Billi99, Bogac18} for details.
	
	For $\mu, \nu \in \mathcal{P}(\mathbb{R})$, the \emph{convolution} $\mu * \nu$ is given by
	\begin{equation*}
		\mu * \nu(B) = \int_{\mathbb{R}} \nu(B-x) \mathrm{d}\mu(x) = \int_{\mathbb{R}} \mu(B-y) \mathrm{d}\nu(y),
	\end{equation*}
	for every Borel subset $B\subseteq\mathbb{R}$. Equivalently, the convolution $\mu*\nu$ is the unique Borel probability measure satisfying
	\begin{equation*}
		\int_{\mathbb{R}} f(x) \mathrm{d}\mu * \nu(x) = \int_{\mathbb{R} \times \mathbb{R}} f(x+y) \mathrm{d}\mu \times \nu(x,y),
	\end{equation*}
	for all $f\in C_{b}(\mathbb{R})$.

 For $\mu\in\mathcal{P}(\mathbb{R})$, its \emph{Fourier transform} is defined by
	\begin{equation*}
		\widehat{\mu}(\xi) = \int_{\mathbb{R}} e^{-2\pi i \xi \cdot x} \mathrm{d}\mu(x).
	\end{equation*}
It is straightforward that
	\begin{equation*}
		\widehat{\mu * \nu}(\xi) = \widehat{\mu}(\xi) \widehat{\nu}(\xi),
	\end{equation*}
	 for all $\xi \in \mathbb{R}$.
	
The following conclusions are useful to study the weak convergence of measures, and we refer readers to \cite{Billi99} for details.
	\begin{lemma} \label{lemma_weak_convergence_equivalent_statement}
		Let $\mu, \mu_{1}, \mu_{2}, \dots \in \mathcal{P}(\mathbb{R})$. Then $\mu_{k}$ converges weakly to $\mu$ if and only if $\lim\limits_{k\to\infty} \widehat{\mu}_{k}(\xi) = \widehat{\mu}(\xi)$ for every $\xi \in \mathbb{R}$.
	\end{lemma}
	
	\begin{lemma}\label{lem_factorization}
		Let $\{\mu_k\}_{k=1}^\f,\{\nu_k\}_{k=1}^\f\subseteq \mathcal{P}(\R)$. If $\mu_k$ and $\nu_k$ converge weakly to $\mu$ and $\nu$ respectively, then we have $\mu_k*\nu_k$ converges weakly to $\mu*\nu$.
	\end{lemma}

	\subsection{Admissible pairs and infinite convolutions}
	 For $a,b\in\mathbb{R}$ with $a\neq 0$, we define a linear transformation $T_{a,b}:\mathbb{R}\to\mathbb{R}$ by
	\begin{equation}\label{equation_transform_on_realline}
		T_{a,b}(x)  = ax+b.
	\end{equation}
	The spectrality of measures is invariant under linear transformations.
\begin{lemma}\label{lemma_invariant_of_spectrality_under_lineartransf}
		If $\mu \in \mathcal{P}(\mathbb{R})$ is a spectral measure with a spectrum $\Lambda$, then the measure $\mu\circ T_{a,b}^{-1}$ is a spectral measure with a spectrum $\frac{1}{a}\Lambda$ for $a,b\in\mathbb{R}$ with $a\neq 0$.
	\end{lemma}

Suppose that $\mu$ is a Borel probability measure on $\R$. Let
	\begin{equation*}
		\mathcal{Z}(\mu)=\{\xi\in\R^d:\widehat{\mu}(\xi+k)=0 \ \text{for all}\ k\in\Z\}.
	\end{equation*}
	We call $\mathcal{Z}(\mu)$ the \emph{integral periodic zero set} of $\mu$. If the infinite convolution $\mu$ is generated by finitely many admissible pairs, the authors in~\cite{LiMiaoWang24} provided the following simple method to show integral periodic zero set of $\mu$ is empty.
\begin{theorem}\label{thm_IPZSE}
Suppose that the sequence of admissible pairs $\{ (N_k, B_k)\}_{k=1}^\f$ is chosen from a finite set of admissible pairs. Let $\mu$ be the infinite convolution given by \eqref{def_ic}. If for each $k \ge 1$,
$$
\mathrm{gcd}\Big( \bigcup_{j=k}^\f (B_j - B_j) \Big) =1,
$$
  then  $\mcal{Z}(\mu)=\emptyset$.
\end{theorem}

Let $\mu_k$ and $\mu$ be given by \eqref{def_ick} and \eqref{def_ic}. We write
$$
\mu_{>k} = \delta_{(N_1 N_2 \cdots N_{k+1})^{-1} B_{k+1} } * \delta_{(N_1 N_2 \cdots N_{k+2})^{-1} B_{k+2} } * \cdots,
$$ and it is clear that $\mu = \mu_k * \mu_{>k}$. We define
\begin{equation}\label{nugeqn}
  \nu_{>k}(\;\cdot\;) = \mu_{>k}\Big( \frac{1}{N_1 N_2 \cdots N_k} \; \cdot\; \Big),
\end{equation}
which is equivalent to
 $$
\nu_{>k}=\delta_{N_{k+1}^{-1} B_{k+1}} * \delta_{(N_{k+1} N_{k+2})^{-1} B_{k+2}} * \cdots.
$$

The integral periodic zero sets  of infinite convolutions are closely related to  equi-positive families, which  are  an important tool to study the spectrality of fractal measures with compact support, see  \cite{An-Fu-Lai-2019,Dutkay-Haussermann-Lai-2019} for details. Then  in~\cite{LiMiaoWang24},  it was   generalized   to the current version for infinite convolutions without compact support,   and it shows that integral periodic zero sets provide a sufficient condition for the spectrality of infinite convolutions.

\begin{theorem}\label{thm_ipzs}
  Given a sequence of admissible pairs $\{(N_k,B_k)\}_{k=1}^\f$, suppose that the infinite convolution $\mu$ defined by \eqref{def_ic} exists, and the sequence $\{ \nu_{>k} \}$ is defined by \eqref{nugeqn}.
  If there exists a subsequence $\{\nu_{>k_j}\}$   convergent weakly to $\nu$ with $\mcal{Z}(\nu)=\emptyset$, then $\mu$ is a spectral measure with a spectrum in $\Z$.
\end{theorem}

	\section{Random Measures}

	\subsection{Characteristic Functions and Baire Functions}
	We recall some definitions and conclusions from probability theory which are used in our proofs. 	
	
	Let $X$ be a non-empty set and  $D$  a non-empty collection of some subsets of $X$.
	We say $D$ is \emph{a $\lambda$-system on $X$} if it satisfies the following
	
	$(\rmnum{1})$ $X \in D$;
	
	$(\rmnum{2})$ if $A, B\in D$ with $A\subseteq B$, then $B\setminus A \in D$;
	
	$(\rmnum{3})$ if $A_{1} \subseteq A_{2} \subseteq A_{3} \subseteq \dots $ is an increasing sequence of sets in $D$, then $\bigcup_{n=1}^{\infty}A_{n} \in D$. \\
	Let $P$ be  a non-empty collection of some subsets of $X$. We say  $P$ is \emph{a $\pi$-system on $X$} if $A \cap B \in P$ for all $A, B\in P$. We write  $\sigma(P) $ for the $\sigma$-algebra generated by $P$. The following conclusion is standard in probability theory, and we refer readers to \cite{Kall} for details.
	\begin{theorem}\label{thm_dynkin's}
		If $P$ is a $\pi$-system and $D$ is a $\lambda$-system with $P \subseteq D$, then $\sigma(P) \subseteq D$.
	\end{theorem}

	To introduce the Baire hierarchy of Borel measurable functions on a metric space, we need some basic facts of ordinals, and we refer readers to \cite{BBT08, Kechris95} for details.
	\begin{theorem}
		There exists an uncountable, well-ordered set $\mathrm{ORD}$ with an order relation $<$ so that \\
		$(\rmnum{1})$ $\mathrm{ORD}$ has a last element denoted by $\omega_1$. \\
		$(\rmnum{2})$ For every $\alpha_0\in \mathrm{ORD}$ with $\alpha_0 \neq \omega_1$, the set $\{\alpha \in \mathrm{ORD} : \alpha <\alpha_0\}$ is countable.\\
		$(\rmnum{3})$There is an element $\omega\in\mathrm{ORD}$  such that
		$$
		\{ \alpha \in \mathrm{ORD} : \alpha<\omega \}=\{0,1,2,3,\ldots\}
		$$
		and $<$ has its usual meaning in the set of nonnegative integers.
	\end{theorem}	
	
	We may regard $\mathrm{ORD} $ as a long list starting with 0 and continuing just until uncountably many elements have been listed:
	$$
	0 <1<2<\cdots <\omega <\omega+1<\omega+2< \cdots <\omega^2<\omega^2+1<\cdots <\omega_1.
	$$
	We call all the elements of $\mathrm{ORD} $ {\it ordinals}. Each element prior to $\omega$ is called a {\it finite ordinal}. Each
	element from then, but prior to the last one $\omega_1$, is called a countable ordinal. The element  $\omega_1$ is
	called the {\it first uncountable ordinal}. Any element that does not have an immediate predecessor is called a {\it limit ordinal}.

	Let $(X,d)$ be a metric space. We write
$$
B_{0}(X)=\{f:X\to [0,1];\ f\ \text{is continuous}\}.
$$
	For each countable ordinal $\alpha\in \mathrm{ORD}  $, we define $B_{\alpha}(X)$ inductively as follows: If $\alpha$ is a successor ordinal, $B_{\alpha}(X)$ is the set of all limits of pointwisely convergent sequences in $B_{\alpha-1\emph{}}(X)$; If $\alpha$ is a limit ordinal, we write  $B_{\alpha}(X) = \bigcup_{\beta < \alpha} B_{\beta}(X)$. Functions in $B_{\alpha}(X)$ are said to be of \emph{Baire class $\alpha$} on $X$.
	Let
	\begin{equation}\label{def_Baire_functions}
		\mathrm{Ba}(X) = \bigcup_{\alpha < \omega_{1}} B_{\alpha}(X).
	\end{equation}	
	The  class $\mathrm{Ba}(X)$  is the smallest set of real-valued functions containing all continuous functions whose ranges are contained in $[0,1]$ and is closed under pointwise convergence. Every element of $\mathrm{Ba}(X)$ is called a \emph{Baire function}. Baire functions were first studied by  Ren\'e Baire\cite{Baire-1899}, and we refer readers to \cite{Folla99, Kechris95, Mauldin74} for details.

	Given  $B \in \mathcal{B}(X)$, we write $\chi_{B} $ for the characteristic function, that is, $\chi_{B} (x)$ is equal to $1$ when $x\in B$ and equal to $0$ when $x\notin B$. The following conclusion shows that 	all characteristic functions generated by Borel sets are Baire functions. This conclusion should be contained in some literature, but we did not find a proper reference for it. For the readers' convenience, we include a proof here.
	\begin{theorem}\label{thm_chi_B_Baire_function}
		Suppose that $(X,d_X)$ is a metric space and $\mathrm{Ba}(X)$ is given by \eqref{def_Baire_functions}. 
		Then $ \chi_{B} \in \mathrm{Ba}(X)$  for all $B \in \mathcal{B}(X)$.
	\end{theorem} 
	
	\begin{proof}
		We  write  $\mathcal{G} = \left\{ B\in\mathcal{B}(X) : \chi_{B} \in \mathrm{Ba}(X)\right\}$, it is sufficient to  show that  $\mathcal{B}(X)\subseteq \mathcal{G}$.
		
		First, we claim that $\mathcal{G}$ contains all open sets. Let $d$ be a metric on $X$ given by
$$
d(x,y) = \min\{ d_X(x, y), 1 \},
$$
 for all $x, y \in X$. Given  a subset $A\subseteq X$, 		for all $x\in X$, we write
		\begin{equation*}
			d(x,A) = \begin{cases}
				\inf_{y\in A} d(x,y), &\text{ if } A\neq \emptyset ;\\
				1, &\text{ if } A=\emptyset.
			\end{cases}
		\end{equation*}
		For every given subset $A$, it is clear that the mapping $x\mapsto d(x, A)$ is continuous  since
		\begin{equation*}
			\vert d(x, A) - d(x', A) \vert \leq d(x, x'),
		\end{equation*}
		for all $x, x' \in X$.
		
		Fix an open set $U\subseteq X$,  and for every integer $n>0$, we write  that
		$$
		F_{n} = \{ x\in X : d(x,U^{c})\geq \frac{1}{n}\}.
		$$
		It is clear that  $F_{n}\subseteq X$ is closed and $F_{n}\subseteq F_{n+1}$ for each $n\geq 1$.
		Let
		\begin{equation*}
			f_{n}(x) = \dfrac{d(x,U^{c})}{d(x,U^{c}) + d(x, F_{n})}.
		\end{equation*}
		It is obvious that $f_{n}$ is a continuous function satisfying $0\leq f_{n} \leq 1$ for each $n$, and  we have $f_n\in B_{0}(X)$. 		Since $\{F_{n}\}_{n=1}^\infty$ is increasing and $U = \bigcup_{n=1}^{\infty} F_{n}$,   the sequence $\{f_{n}(x)\}_{n=1}^\infty$ converges pointwisely and monotonically to $\chi_{U}$. This implies that $\chi_{U}\in  \mathrm{Ba}(X)$ for every open subset $U\subset X$, hence $\mathcal{G}$ contains all open subset of $X$.

		Next, we claim  that  $\mathcal{G}$ is a $\lambda$-system.  It is equivalent to verify the three conditions in the definition of $\lambda$-system.
		
		$(\rmnum{1})$ It is clear that  $X \in \mathcal{G}$ since $\mathcal{G}$ contains all open sets.
		
		$(\rmnum{2})$ Given $A, B \in \mathcal{G}$ such that $A \subseteq B$, let  $\{f_{n}\}_{n=1}^{\infty}$ and $\{g_{n}\}_{n=1}^{\infty}$ be the sequences of continuous functions pointwisely convergent  to $\chi_{A}$ and $\chi_{B}$, respectively. Obviously, $\{g_{n} - f_{n}\}_{n=1}^{\infty}$ is a sequence of continuous functions pointwisely convergent to $\chi_{B\setminus A}$. Since  $\{\min\{ \max \{g_{n} - f_{n}, 0\}, 1\}\}_{n=1}^{\infty}$ is still a sequence of continuous functions  pointwisely  convergent to $\chi_{B\setminus A}$, we have  $\chi_{B\setminus A}\in  \mathrm{Ba}(X)$, and it implies that   $B\setminus A \in \mathcal{G}$.
		
		$(\rmnum{3})$ Suppose that $\{A_{n}\}_{n=1}^{\infty}$ is an increasing sequence of subsets in $\mathcal{G}$. For each $n\geq 1$,  there exists a sequence of continuous functions $\{f_{n,m}\}_{m=1}^{\infty}$  pointwisely convergent to $\chi_{A_{n}}$ with $f_{n,m}(X) \subseteq [0,1]$. We write
		\begin{equation*}
			f_{n}(x) = \max_{1\leq i \leq n} f_{i,n}(x),
		\end{equation*}
and it is clear that $f_n$ is continuous with  $f_{n}(X) \subseteq [0,1]$, and the sequence $\{f_{n}\}_{n=1}^{\infty}$ converges pointwisely to $\chi_{\cup_{n=1}^{\infty}A_{n}}$. Hence $\chi_{\cup_{n=1}^{\infty}A_{n}}\in \mathrm{Ba}(X) $, and it follows that   $\bigcup_{n=1}^{\infty}A_{n} \in \mathcal{G}$.
		
		Therefore $\mathcal{G}$ is a $\lambda$-system. Since all open subsets in $X$ form a $\pi$-system and $\mathcal{B}(X) = \sigma(\{ U: U\subseteq  X\text{ is open} \})$, by Theorem~\ref{thm_dynkin's}, we obtain that
		$$
		\mathcal{B}(X) = \sigma(\{ U: U\subseteq  X \text{ is open} \}) \subseteq \mathcal{G},
		$$
		which completes the proof.
	\end{proof}

	\subsection{Borel Measurability of Random Measures}
	\label{sec_measurability}
	Let $\mathcal{P}(\mathbb{R})$ denote the collection of all Borel probability measures on $\mathbb{R}$, and let $\mathcal{T}_{w}$ be the weak topology on $\mathcal{P}(\mathbb{R})$.

	Given a sequence $\{(N_k,B_k)\}_{k=1}^\f$ where $N_k\ge2$, $B_k\subseteq [0,+\f)$ is finite for all $k>0$ and a sequence of positive integers $\bn=\{ n_{k}\}_{k=1}^\f$,  recall that the mappings   $\mu_k^{\mathbf{n}}$  and $\mu^{\mathbf{n}}$  are defined respectively by
	\begin{equation}\label{Redef_rmk}
		\mu_{k}^{\mathbf{n}}(\boldsymbol{\omega}, B) = \mu_{k}^{\mathbf{n}}(\boldsymbol{\omega})(B) = \delta_{N_{\omega_{1}}^{-n_{1}}B_{\omega_{1}}} * \delta_{N_{\omega_{1}}^{-n_{1}}N_{\omega_{2}}^{-n_{2}}B_{\omega_{2}}} *  \dots * \delta_{N_{\omega_{1}}^{-n_{1}}N_{\omega_{2}}^{-n_{2}}\cdots N_{\omega_{k}}^{-n_{k}}B_{\omega_{k}}}(B),
	\end{equation}
	and
	\begin{equation} \label{Redef_rm}
		\mu^{\mathbf{n}}(\boldsymbol{\omega}, B)= \mu^{\mathbf{n}}(\boldsymbol{\omega})(B)=\delta_{N_{\omega_{1}}^{-n_{1}}B_{\omega_{1}}} * \delta_{N_{\omega_{1}}^{-n_{1}}N_{\omega_{2}}^{-n_{2}}B_{\omega_{2}}} * \dots (B),
	\end{equation}
	for all $\boldsymbol{\omega} \in \Omega$ and all Borel sets $B \in \mathcal{B}(\mathbb{R})$. We assume that $\mu(\boldsymbol{\omega})$ exists for every $\boldsymbol{\omega}\in\Omega$ and recall that
	\begin{equation} \label{Redef_setPhi}
		\Phi = \{ \mu(\boldsymbol{\omega}): \boldsymbol{\omega} \in \Omega \}.
	\end{equation}	

We define a mapping $\phi: (\Omega, d) \to (\mathcal{P}(\mathbb{R}), \mathcal{T}_{w})$ by
	\begin{equation} \label{def_phi}
		\phi (\boldsymbol{\omega}) = \mu^{\mathbf{n}}(\boldsymbol{\omega}),
	\end{equation}
	where $\mu^{\mathbf{n}}(\boldsymbol{\omega})$ is given by \eqref{Redef_rm}. This mapping plays an important role in the measurability of random measures, and we show it is continuous in the next conclusion.
	\begin{lemma}\label{lemma_continuous_mapping_from_symbol_space_to_Prob_space}
		 If $\Phi$ is tight, then $\phi$ given by \eqref{def_phi} is continuous.
	\end{lemma}
	
	\begin{proof}
		Given $\boldsymbol{\omega}=\omega_{1}\omega_{2}\ldots\omega_{j}\ldots\in\Omega$, it is equivalent to prove that for every sequence $\{\boldsymbol{\omega}^{(k)}=\omega_{1}^{(k)}\omega_{2}^{(k)}\ldots\omega_{j}^{(k)}\ldots \in\Omega\}_{k=1}^{\infty}$ convergent to $\boldsymbol{\omega}$, we have that $\lim_{k\to\infty }\phi(\boldsymbol{\omega}^{k})=\phi(\boldsymbol{\omega}).$
		By Lemma \ref{lemma_weak_convergence_equivalent_statement}, it is sufficient to show that
		$$
		\lim\limits_{k\to\infty} \widehat{\mu^{\mathbf{n}}(\boldsymbol{\omega}^{(k)})}(\xi) = \widehat{\mu^{\mathbf{n}}(\boldsymbol{\omega})} (\xi).
		$$
		for every $\xi \in \mathbb{R}$ and $\xi \neq 0$ since $\widehat{\mu^{\mathbf{n}}(\boldsymbol{\omega}^{(k)})}(0) = \widehat{\mu^{\mathbf{n}}(\boldsymbol{\omega})}(0) = 1.$

Fix $\xi \neq 0$.  Arbitrarily choose $\epsilon$ such that $0<\epsilon<\frac{1}{4}$. Since $\Phi$ is tight, there exists $M>0$ such that
		\begin{equation}\label{phitight}
			\nu([-M,M]) > 1-\epsilon,
		\end{equation}
		for all $\nu \in \Phi$. Choose an integer $m>0$ such that
		\begin{equation}\label{equation_assumption_for_m_wrt_epsilon}
			\frac{1}{2^m} < \dfrac{\epsilon}{\vert \xi M\vert}.
		\end{equation}
		 Since $\lim_{k\to \infty}\boldsymbol{\omega}^{(k)}=\boldsymbol{\omega} $, there exists $K>0$ such that $d(\boldsymbol{\omega}^{(k)},\boldsymbol{\omega})<\frac{1}{2^m}$  for all integers $k>K$.
		This   implies  that for all $k> K$,
		$$
		\omega_{j}^{(k)} = \omega_{j}, \qquad  \textit{ for $j=1,2,\ldots, m$.}
		$$

For each integer $k>0$, we write
	\begin{equation}
		\mu^{\mathbf{n}}_{>k}(\boldsymbol{\omega}) =\delta_{ N_{\omega_{1}}^{-n_{1}} N_{\omega_{2}}^{-n_{2}} \dots N_{\omega_{k+1}}^{-n_{k+1}} B_{\omega_{k+1}}} * \delta_{N_{\omega_{1}}^{-n_{1}} N_{\omega_{2}}^{-n_{2}} \dots N_{\omega_{k+2}}^{-n_{k+2}} B_{\omega_{k+2}}} *  \dots,
	\end{equation}
and it is clear that $\mu^{\mathbf{n}}(\boldsymbol{\omega})= \mu^{\mathbf{n}}_{k}(\boldsymbol{\omega}) * \mu^{\mathbf{n}}_{>k}(\boldsymbol{\omega})$. For simplicity, we  write
\begin{equation}\label{nu-large-than-n}
		\nu^{\mathbf{n}}_{>k}(\boldsymbol{\omega})(\;\cdot\;) = \mu^{\mathbf{n}}_{>k}(\boldsymbol{\omega})\left( \frac{1}{N_{\omega_{1}}^{n_{1}} N_{\omega_{2}}^{n_{2}} \cdots N_{\omega_{n}}^{n_{n}}} \; \cdot\; \right),
\end{equation}
and it is equivalent to
	\begin{equation}
		\nu^{\mathbf{n}}_{>k}(\boldsymbol{\omega}) =\delta_{ N_{\omega_{k+1}}^{-n_{k+1}} B_{\omega_{k+1}}} * \delta_{N_{\omega_{k+1}}^{-n_{k+1}}N_{\omega_{k+2}}^{-n_{k+2}} B_{\omega_{k+2}}} *  \cdots.
	\end{equation}

		Hence for $k> K$, we have that
		\begin{eqnarray*}
			\vert \widehat{\mu^{\mathbf{n}}(\boldsymbol{\omega}^{(k)})}(\xi) - \widehat{\mu^{\mathbf{n}}(\boldsymbol{\omega})}(\xi)\vert
			&\leq& \vert \widehat{\mu^{\mathbf{n}}_{m}(\boldsymbol{\omega})}(\xi)\vert \cdot \vert \widehat{\mu^{\mathbf{n}}_{>m}(\boldsymbol{\omega}^{(k)})}(\xi) - \widehat{\mu^{\mathbf{n}}_{>m}(\boldsymbol{\omega})}(\xi)\vert\\
			&\leq& \vert \widehat{\mu^{\mathbf{n}}_{>m}(\boldsymbol{\omega}^{(k)})}(\xi) - \widehat{\mu^{\mathbf{n}}_{>m}(\boldsymbol{\omega})}(\xi)\vert\\
			&\leq& \Big\vert \int_{\mathbb{R}} e^{-2\pi i \xi \cdot x} \mathrm{d}\mu^{\mathbf{n}}_{>m}(\boldsymbol{\omega}^{(k)}) - \int_{\mathbb{R}} e^{-2\pi i \xi \cdot x} \mathrm{d}\mu^{\mathbf{n}}_{>m}(\boldsymbol{\omega}) \Big\vert \\
			&\leq& \bigg\vert \int_{\mathbb{R}} \cos(2\pi \xi x) \mathrm{d}\mu^{\mathbf{n}}_{>m}(\boldsymbol{\omega}^{(k)}) - \int_{\mathbb{R}} \cos(2\pi \xi x) \mathrm{d}\mu^{\mathbf{n}}_{>m}(\boldsymbol{\omega}) \bigg\vert  \\
			&& +	\bigg\vert \int_{\mathbb{R}} \sin(2\pi \xi x) \mathrm{d}\mu^{\mathbf{n}}_{>m}(\boldsymbol{\omega}^{(k)}) - \int_{\mathbb{R}} \sin(2\pi \xi x) \mathrm{d}\mu^{\mathbf{n}}_{>m}(\boldsymbol{\omega}) \bigg\vert.
		\end{eqnarray*}
		
		Note that for every $\boldsymbol{\beta}=\beta_1\beta_2\ldots\beta_j\ldots \in \Omega$, we have that for all $b_{\beta_{j}} \in B _{\beta_{j}}$ and all $l\geq 1$,
		\begin{equation*}
			\sum_{j=l+1}^{\infty} \dfrac{b_{\beta_{j}}}{N_{\beta_{l+1}}^{n_{l+1}}N_{\beta_{l+2}}^{n_{l+2}} \dots N_{\beta_{j}}^{n_{j}}} \leq \sum_{j=l+1}^{\infty} \dfrac{b_{\beta_{j}}}{N_{\beta_{l+1}}N_{\beta_{l+2}} \dots N_{\beta_{j}}},
		\end{equation*}
which implies
        \begin{equation*}
				\nu^{\mathbf{n}}_{>l}(\boldsymbol{\beta})([-M,M]) \geq \nu_{>l}(\boldsymbol{\beta})([-M,M]).
			\end{equation*}
			Since $\nu_{>l}(\boldsymbol{\beta}) \in \Phi$ and $N_k\ge2$ for all $k>0$, by \eqref{phitight} and \eqref{nu-large-than-n}, it follows  that
			\begin{equation*}
				\begin{aligned}
					\mu^{\mathbf{n}}_{>l}(\boldsymbol{\beta})\bigg(\dfrac{1}{2^{l}}[-M,M]\bigg) &\geq
					\mu^{\mathbf{n}}_{>l}(\boldsymbol{\beta})(N_{\beta_{1}}^{-n_{1}}N_{\beta_{2}}^{-n_{2}}\dots N_{\beta_{l}}^{-n_{l}}[-M,M])\\
					&= \nu^{\mathbf{n}}_{>l}(\boldsymbol{\beta})([-M,M])\\
					&\geq \nu_{>l}(\boldsymbol{\beta})([-M,M]) \\
                    &>1-\epsilon,
				\end{aligned}
		\end{equation*} 
		for all $l\geq 1$. Note that this implies that  the measure of $\mu^{\mathbf{n}}_{>l}(\boldsymbol{\beta})$ is concentrated on some neighborhood of $x=0$, and the radius of the neighborhood converges to $0$ as $l$ tends to $\infty$,  which is independent of the choice of $\boldsymbol{\beta} \in \Omega$.
		
		Letting $E_{m} = \frac{1}{2^{m}}[-M, M]$. By  \eqref{equation_assumption_for_m_wrt_epsilon}, it is clear that  $\vert \frac{2\pi \xi \cdot M}{2^{m}} \vert <\frac{\pi}{2}$. Hence, both values of the integrals $\int_{E_{m}} \cos(2\pi \xi x) \mathrm{d}\mu^{\mathbf{n}}_{>m}(\boldsymbol{\omega}^{(k)})$ and $\int_{E_{m}} \cos(2\pi \xi x) \mathrm{d}\mu^{\mathbf{n}}_{>m}(\boldsymbol{\omega})$ are contained in the interval $[\cos(\frac{\pi \xi \cdot M}{2^{m-1}}), 1]$, and this  implies that
		\begin{equation*}
			\bigg\vert \int_{E_{m}} \cos(2\pi \xi x) \mathrm{d}\mu^{\mathbf{n}}_{>m}(\boldsymbol{\omega}^{(k)}) - \int_{E_{m}} \cos(2\pi \xi x) \mathrm{d}\mu^{\mathbf{n}}_{>m}(\boldsymbol{\omega}) \bigg\vert \leq 1-\cos(\frac{\pi \xi \cdot M}{2^{m-1}}) \leq \dfrac{(\pi \xi M)^{2}}{2^{2m-1}}.
		\end{equation*}
		Therefore, combining it with  \eqref{equation_assumption_for_m_wrt_epsilon},  we obtain that
		\begin{eqnarray*}
			&&\bigg\vert \int_{\mathbb{R}} \cos(2\pi \xi x) \mathrm{d}\mu^{\mathbf{n}}_{>m}(\boldsymbol{\omega}^{(k)}) - \int_{\mathbb{R}} \cos(2\pi \xi x) \mathrm{d}\mu^{\mathbf{n}}_{>m}(\boldsymbol{\omega}) \bigg\vert \\
			& \leq& \bigg\vert \int_{E_{m}} \cos(2\pi \xi x) \mathrm{d}\mu^{\mathbf{n}}_{>m}(\boldsymbol{\omega}^{(k)}) - \int_{E_{m}} \cos(2\pi \xi x) \mathrm{d}\mu^{\mathbf{n}}_{>m}(\boldsymbol{\omega}) \bigg\vert\\
			&& \hspace{3cm} + \bigg\vert \int_{\mathbb{R} \setminus E_{m}} \cos(2\pi \xi x) \mathrm{d}\mu^{\mathbf{n}}_{>m}(\boldsymbol{\omega}^{(k)})\bigg\vert + \bigg\vert \int_{\mathbb{R} \setminus E_{m}} \cos(2\pi \xi x) \mathrm{d}\mu^{\mathbf{n}}_{>m}(\boldsymbol{\omega}) \bigg\vert \\
			&<& \dfrac{(\pi \xi M)^{2}}{2^{2m-1}} + 2\epsilon\\
			&\leq& C_1\epsilon,
		\end{eqnarray*}
		where $C_1>0$ is a constant.  Similarly, we have that
		\begin{equation*}
			\bigg\vert \int_{E_{m}} \sin(2\pi \xi x) \mathrm{d}\mu^{\mathbf{n}}_{>m}(\boldsymbol{\omega}^{(k)}) - \int_{E_{m}} \sin(2\pi \xi x) \mathrm{d}\mu^{\mathbf{n}}_{>m}(\boldsymbol{\omega}) \bigg\vert \leq  \dfrac{\vert \pi \xi M \vert}{2^{m}},
		\end{equation*}
		and it follows that		
		\begin{equation*}
			\bigg\vert \int_{\mathbb{R}} \sin(2\pi \xi x) \mathrm{d}\mu^{\mathbf{n}}_{>m}(\boldsymbol{\omega}^{(k)}) - \int_{\mathbb{R}} \sin(2\pi \xi x) \mathrm{d}\mu^{\mathbf{n}}_{>m}(\boldsymbol{\omega}) \bigg\vert < \dfrac{\vert \pi \xi M \vert}{2^{m}} + 2\epsilon \leq C_2 \epsilon,
		\end{equation*}
		where $C_2>0$ is a constant.
		
		Consequently, for all $k>K$, we obtain that
		\begin{equation*}
			\vert \widehat{\mu^{\mathbf{n}}(\boldsymbol{\omega}^{(k)})}(\xi) - \widehat{\mu^{\mathbf{n}}(\boldsymbol{\omega})}(\xi)\vert <  C\epsilon,	\end{equation*}
		where $C>0$ is a  constant, and it follows that
		\begin{equation*}
			\lim\limits_{k\to\infty} \widehat{\mu^{\mathbf{n}}(\boldsymbol{\omega}^{(k)})}(\xi) = \widehat{\mu^{\mathbf{n}}(\boldsymbol{\omega})}(\xi),
		\end{equation*}
		for all $\xi\in\mathbb{R}$. Hence the mapping $\mu^{\mathbf{n}}$ is continuous.
	\end{proof}
	
	Now, we are ready to prove the measurability of the mapping $\mu^{\mathbf{n}}$ with respect to $\boldsymbol{\omega}$.
	\begin{proposition}\label{prop_Borel_measurable_mapping}
If $\Phi$ is tight, then for each $B \in \mathcal{B}(\mathbb{R})$, 	 the mapping $\mu^{\mathbf{n}}(B)$ given by \eqref{Redef_rm} is $\mathcal{F}$-measurable.
	\end{proposition}	
	\begin{proof}
		Fix $B \in \mathcal{B}(\mathbb{R})$, and  we define the mapping $\pi_{B}: (\mathcal{P}(\mathbb{R}), \mathcal{T}_{w}) \to (\mathbb{R},\mathcal{B}(\mathbb{R}))$ by
		\begin{equation*}
			\pi_{B}(\eta) = \eta(B).
		\end{equation*}
		It immediately follows that  the mapping $\mu^{\mathbf{n}}(B)$ is the composition of $\pi_{B}$ and $\phi$, that is,
		\begin{equation*}
			\mu^{\mathbf{n}}(\boldsymbol{\omega},B) = \pi_{B} \circ \phi(\boldsymbol{\omega}),
		\end{equation*}
		where $\phi$ is defined by \eqref{def_phi}.

Since $\Phi$ is tight, by Lemma \ref{lemma_continuous_mapping_from_symbol_space_to_Prob_space}, it is clear that $\phi$ is continuous,   and  it suffices to show that $\pi_{B}$ is a Borel mapping, that is,
		\begin{equation*}
			\{\eta\in\mathcal{P}(\mathbb{R}) : \eta(B) < t\} \in \sigma(\mathcal{T}_{w}),
		\end{equation*}
		for all $t >0$. Note that $\{\eta\in\mathcal{P}(\mathbb{R}) : \eta(B) < t\} =\emptyset $ for all $t\leq 0$.
		
For every Baire function $f \in \mathrm{Ba}(\mathbb{R})$, we define a mapping  $\Lambda_{f}: \mathcal{P}(\mathbb{R}) \to \mathbb{R}$ by
\begin{equation}\label{def_lnrfun}
		\Lambda_{f}(\eta) = \int f \mathrm{d}\eta,
\end{equation}
		and it is well-defined since every function $f \in \mathrm{Ba}(\mathbb{R})$ is bounded and thus integrable under $\eta \in \mathcal{P}(\mathbb{R})$. Recall that $\mathcal{T}_{w}$ is the weak topology of $\mathcal{P}(\mathbb{R})$, that is, $\mathcal{T}_{w}$ is the coarsest topology for which $\{\Lambda_{f}\}_{f\in C_{b}(\mathbb{R})}$ is continuous.
		
		For every $f \in \mathrm{Ba}(\mathbb{R})$, we write that for every $t\in\mathbb{R}$,
		\begin{equation}\label{def_Mft}
			\mathcal{M}_{f}(t) = \{\eta\in\mathcal{P}(\mathbb{R}) : \Lambda_{f}(\eta) < t\}.
		\end{equation}
		For every ordinal $\alpha \in [0, \omega_{1})$, we denote by $P(\alpha)$ the property that
		\begin{center}
			$P(\alpha): \quad \mathcal{M}_{f}(t) \in \sigma(\mathcal{T}_{w})$ for  every $t \in\mathbb{R}$ and  every $f \in B_{\alpha}(\mathbb{R})$.
		\end{center}
		We claim that $P(\alpha)$ is true for all ordinals  $\alpha \in [0, \omega_{1})$.
		
		We prove it by transfinite induction. First,  we show $P(\alpha)$ is true for $\alpha =0 $. Arbitrarily choose $f \in B_{0}(\mathbb{R})$, and it is clear that   $\Lambda_{f}$ is continuous since $B_{0}(\mathbb{R})\subset C_{b}(\mathbb{R})$. Hence we have that
		\begin{equation*}
			\mathcal{M}_{f} (t) = \{\eta\in\mathcal{P}(\mathbb{R}) : \Lambda_{f}(\eta) < t\} \in \mathcal{T}_{w}\subseteq \sigma(\mathcal{T}_{w}),
		\end{equation*}
		for all $t\in\mathbb{R}$, and the property $P(0)$ is true.
		
		Next, for an  ordinal $\alpha < \omega_{1}$, we assume that $P(\alpha)$ is true, that is $\mathcal{M}_{f}(t) \in \sigma(\mathcal{T}_{w})$ for  every $t \in\mathbb{R}$ and  every $f \in B_{\alpha}(\mathbb{R})$.
		
		For each $f \in B_{\alpha+1}(\mathbb{R})$, there exists a sequence of functions $\{f_{n}\}_{n=1}^{\infty}$ in $B_{\alpha}(\mathbb{R})$ convergent pointwisely to $f$. Since  $f_{n} \in B_{\alpha}(\mathbb{R})$ is bounded by $1$, we have that $f_{n} \in L^{1}(\eta)$ for all $\eta \in \mathcal{P}(\mathbb{R})$. Applying the dominated convergence theorem, we obtain
		\begin{equation*}
			\lim\limits_{n\to\infty} \int f_{n} \mathrm{d}\eta = \int \lim\limits_{n\to\infty}f_{n} \mathrm{d}\eta
			= \int f \mathrm{d}\eta = \Lambda_{f}(\eta),
		\end{equation*}
		for all $\eta \in \mathcal{P}(\mathbb{R})$.
		It follows that for every $t \in\mathbb{R}$,
		\begin{eqnarray*} 
			\mathcal{M}_{f}(t) &=& \{\eta\in\mathcal{P}(\mathbb{R}) : \lim\limits_{n\to\infty} \int f_{n} \mathrm{d}\eta < t\}  \\
			&=& \bigcup_{k=1}^{\infty} \bigcup_{N=1}^{\infty} \bigcap_{n=N}^{\infty} \left\{\eta\in\mathcal{P}(\mathbb{R}) : \Lambda_{f_{n}}(\eta) < t - \frac{1}{k}\right\}\\
			&=&\bigcup_{k=1}^{\infty} \bigcup_{N=1}^{\infty} \bigcap_{n=N}^{\infty}  \mathcal{M}_{f_n}(t-\frac{1}{k}).
		\end{eqnarray*}
		
		Since $\mathcal{M}_{f_n}(t-\frac{1}{k}) \in \sigma(\mathcal{T}_{w})$ for all $k>0$ and all $n>0$,  we immediate have $\mathcal{M}_{f} (t)\in \sigma(\mathcal{T}_{w})$ for all $t\in \mathbb{R}$ and all  $f\in B_{\alpha+1}(\mathbb{R}) $. Hence the property  $P(\alpha+1)$ is true.
		
		Finally, for a nonzero limit ordinal $\alpha < \omega_{1}$, suppose that $P(\lambda)$ holds for all $\lambda < \alpha$. Since $\alpha$ is a limit ordinal, for each  function $f \in B_{\alpha}(\mathbb{R})$, it follows from the definition of $B_{\alpha}(\mathbb{R})$ that there exists an ordinal $ \lambda_{0} < \alpha < \omega_{1}$ such that $f \in B_{ \lambda_{0}}(\mathbb{R})$. Hence for all $t\in \mathbb{R}$, we have that  $\mathcal{M}_{f}(t) \in \sigma(\mathcal{T}_{w})$ by induction hypothesis, and $P(\alpha)$ is true. Therefore the claim holds.
		
		By the claim, we have that  $\mathcal{M}_{f}(t) \in \sigma(\mathcal{T}_{w})$  for all $t\in\mathbb{R}$ and all $f\in \mathrm{Ba}(\mathbb{R})$.   For each $B\in\mathcal{B}(\mathbb{R})$, by Theorem \ref{thm_chi_B_Baire_function}, we have $\chi_B\in \mathrm{Ba}(\mathbb{R})$. By \eqref{def_lnrfun} and \eqref{def_Mft}, it follows that
		\begin{equation*}
			\{\eta\in\mathcal{P}(\mathbb{R}) : \eta(B) < t\}= \mathcal{M}_{\chi_B}(t)\in \sigma(\mathcal{T}_{w}),
		\end{equation*}
		for all $t > 0$, and $\pi_{B}$ is a Borel mapping,  which completes the proof.
	\end{proof}
	
In \cite{Li-Miao-Wang-2021}, Li, Miao and Wang provided a necessary and sufficient condition for the convergence of infinite convolutions.
	\begin{theorem}\label{thm_weak-convergence-infinite-convolution}
		Let $\{ A_k\}_{k=1}^\f$ be a sequence of finite subsets of $(0,+\f)$ satisfying that $\# A_k \ge 2$ for each $k \ge 1$.
		Let $\nu_n$ be defined in \eqref{discrete-convolution}.
		Then the sequence of convolutions $\set{\nu_n}_{n=1}^{\infty}$ converges weakly to a Borel probability measure if and only if
		\begin{equation}\label{thm_existence}
			\sum_{k=1}^{\f} \frac{1}{\# A_k} \sum_{a\in A_k} \frac{a}{1+a} < \f.
		\end{equation}
	\end{theorem}
	
In fact, the  conclusion still holds if   $A_k \subset [0,+\f) $ is finite such that $\# A_k \ge 2$ for each $k \ge 1$. Recall that if $\mathbf{n}$ constantly equals 1, that is $n_k=1$ for all $k\ge 1$, we write that $\mu (\boldsymbol{\omega})=\mu^{\mathbf{n}}(\boldsymbol{\omega})$.

	\begin{proposition}\label{prop_tightness_guarantees_existence}
		
		 Given $\boldsymbol{\omega} \in \Omega$, if $\mu (\boldsymbol{\omega})$ exists, then $\mu^{\mathbf{n}}(\boldsymbol{\omega})$ exists for all sequence $\mathbf{n}$ of positive integers.
	\end{proposition}
	\begin{proof}  
		Since $\mu(\boldsymbol{\omega})$ exists, by \eqref{def_corresponding_P_of_random_measure} and Theorem \ref{thm_weak-convergence-infinite-convolution}, we have that
		\begin{equation}\label{eq_series_convergence_equivalent_to_mu_exist}
			\sum_{k=1}^{\infty} \dfrac{1}{\# B_{\omega_{k}}} \sum_{b \in B_{\omega_{k}}} \dfrac{b}{N_{\omega_{1}}N_{\omega_{2}}\dots N_{\omega_{k}}+b} < \infty.
		\end{equation}
		Since every element in $\mathbf{n}$ is a positive integer, it is clear that
		\begin{equation*}
			\sum_{b \in B_{\omega_{k}}} \dfrac{b}{N_{\omega_{1}}^{n_{1}}N_{\omega_{2}}^{n_{2}}\dots N_{\omega_{k}}^{n_{k}}+b} \leq 		\sum_{b \in B_{\omega_{k}}} \dfrac{b}{N_{\omega_{1}}N_{\omega_{2}}\dots N_{\omega_{k}}+b}.
		\end{equation*}
		It follows that
		\begin{equation*}
			\sum_{k=1}^{\infty} \dfrac{1}{\# B_{\omega_{k}}} \sum_{b \in B_{\omega_{k}}} \dfrac{b}{N_{\omega_{1}}^{n_{1}}N_{\omega_{2}}^{n_{2}}\dots N_{\omega_{k}}^{n_{k}}+b}<\infty,
		\end{equation*}
		and it  implies $\mu^{\mathbf{n}}(\boldsymbol{\omega})$ exists for all $\mathbf{n}$ by Theorem \ref{thm_weak-convergence-infinite-convolution}.
	\end{proof}

Finally, we are ready to prove that the mapping $\mu^\bn$ given by \eqref{Redef_rm} is a random measure, which is a direct consequence of above conclusions.

	\begin{proof}[Proof of Theorem \ref{thm_measurable}]
Since $\mu (\boldsymbol{\omega})$ exists for every $\boldsymbol{\omega} \in \Omega$, by Proposition \ref{prop_tightness_guarantees_existence},   the infinite convolution $\mu^{\mathbf{n}}(\boldsymbol{\omega})$ exists and is a Borel probability measure for all sequence $\mathbf{n}$ of positive integers.  Since $\Phi$ is tight,  by Proposition \ref{prop_Borel_measurable_mapping}, the mapping $\mu^{\mathbf{n}}(B):\Omega\to \R$ given by \eqref{Redef_rm} is $\mathcal{F}$-measurable for each $B \in \mathcal{B}(\mathbb{R})$. Therefore, the mapping $\mu^{\mathbf{n}} : \Omega\times \mathcal{B}(\mathbb{R}) \to [0,+\f]$  is a random measure.
\end{proof}

 \begin{proof}[Proof of Theorem \ref{thm_measurable'}]
The argument is similar to the proof of Theorem \ref{thm_measurable} but simpler, and we omit it.
\end{proof}

	\section{Spectrality of Random Measure}\label{sec_st}
	In this section,  we explore the spectrality of random measures. 	Let $\{(N_k,B_k)\}_{k=1}^\f$ be a  sequence of admissible pairs and $\bn=\{ n_{k}\}_{k=1}^\f$ be a sequence of positive integers.  Let $\mu^{\mathbf{n}}$ be the mapping  given by \eqref{def_random_measure} associated with $\{ (N_{k}, B_{k} ) \}_{k=1}^{\infty}$.

 We assume that $	\Phi = \{ \mu(\boldsymbol{\omega}): \boldsymbol{\omega} \in \Omega \}$ is tight.
 By Theorem \ref{thm_measurable},  $\mu^{\mathbf{n}}$ is a random measure. We rewrite $\mu^{\mathbf{n}}(\boldsymbol{\omega})= \mu^{\mathbf{n}}_{k}(\boldsymbol{\omega}) * \mu^{\mathbf{n}}_{>k}(\boldsymbol{\omega})$ where
	\begin{equation}
		\mu^{\mathbf{n}}_{>k}(\boldsymbol{\omega}) =\delta_{ N_{\omega_{1}}^{-n_{1}} N_{\omega_{2}}^{-n_{2}} \dots N_{\omega_{k+1}}^{-n_{k+1}} B_{\omega_{k+1}}} * \delta_{N_{\omega_{1}}^{-n_{1}} N_{\omega_{2}}^{-n_{2}} \dots N_{\omega_{k+2}}^{-n_{k+2}} B_{\omega_{k+2}}} *  \dots,
	\end{equation}
	is the tail of the infinite convolution, and we define
	\begin{equation}\label{def_nu_words}
		\nu^{\mathbf{n}}_{>k}(\boldsymbol{\omega}) =\delta_{ N_{\omega_{k+1}}^{-n_{k+1}} B_{\omega_{k+1}}} * \delta_{N_{\omega_{k+1}}^{-n_{k+1}}N_{\omega_{k+2}}^{-n_{k+2}} B_{\omega_{k+2}}} *  \cdots.
	\end{equation}
		Note that $\mu^{\mathbf{n}}_{>k}$ and $\nu^{\mathbf{n}}_{>k}$ are also   random measures. If $\bn=\{ 1 \}_{k=1}^\f$, we write
	\begin{equation} \label{def_nu1geqk}
		\nu_{>k}(\boldsymbol{\omega})=\nu^{\mathbf{n}}_{>k}(\boldsymbol{\omega}) =\delta_{ N_{\omega_{k+1}}^{-1} B_{\omega_{k+1}}} * \delta_{N_{\omega_{k+1}}^{-1}N_{\omega_{k+2}}^{-1} B_{\omega_{k+2}}} *  \cdots.
	\end{equation}
	for each $\boldsymbol{\omega}\in\Omega$.

	We first prove that $\mu^{\mathbf{n}}$ is a spectral random measure for unbounded sequence $\mathbf{n}$, that is for every $\boldsymbol{\omega} \in \Omega$, the realization $\mu^{\mathbf{n}}(\boldsymbol{\omega})$ is a spectral measure.
\begin{proof}[Proof of Theorem $\ref{thm_ub}$]
Let  $\delta_0$ denote the  Dirac measure. It is clear that the integral periodic zero set of $\delta_0$ is empty, that is, $\mathcal{Z}(\delta_0)=\emptyset$. For each fixed unbounded sequence $\mathbf{n}=\{n_{k}\}$ and each fixed $\boldsymbol{\omega} \in \Omega$, we  prove that $\{\nu_{>k}^{\mathbf{n}}(\boldsymbol{\omega})\}$ converges weakly to  $\delta_{0}$.  Since $\widehat{\delta_{0}}\equiv1$, by Lemma \ref{lemma_weak_convergence_equivalent_statement}, it is equivalents to prove that $\lim_{k\to \infty }\widehat{\nu_{>k}^\mathbf{n}(\boldsymbol{\omega})}(\xi)=1$ for all $\xi\in\R$.
		
 For every $\epsilon \in (0,1)$, since $\Phi$ is tight,  there exists  a real $M \geq 1$ such that
		$$
		\nu([-M, M]) > 1 - \epsilon,
		$$
		for every $\nu \in \Phi$.  Since the sequence $\{n_{k}\}$ is unbounded, it has a subsequence divergent to infinity, and for simplicity,  we assume that
		\begin{equation*}
			\lim\limits_{k\to\infty} n_{k} = \infty.
		\end{equation*}
Fix $\xi \in \mathbb{R}$. There exists an integer $K>0$ such that for all $k\geq K$,
\begin{equation} \label{ineq_xiep}
		\frac{M\pi \vert \xi \vert }{2^{n_{k+1}-2}} < \epsilon .
\end{equation}

	Moreover, for each  $k\geq 1$, since
		\begin{equation*}
			\sum_{j=k+1}^{\infty} \dfrac{b_{\omega_{j}}}{N_{\omega_{k+1}} N_{\omega_{k+2}}^{n_{k+2}} \dots N_{\omega_{j}}^{n_{j}}} \leq \sum_{j=k+1}^{\infty} \dfrac{b_{\omega_{j}}}{N_{\omega_{k+1}} N_{\omega_{k+2}} \dots N_{\omega_{j}}},
		\end{equation*}
for all $b_{\omega_{j}} \in B_{\omega_{j}}$,		we obtain that
		\begin{equation*}
				\nu_{>k}^{\mathbf{n}}(\boldsymbol{\omega})\Big(N_{\omega_{k+1}}^{1-n_{k+1}} \cdot [-M, M]\Big) \geq \nu_{>k}(\boldsymbol{\omega})([-M, M]).
		\end{equation*}
Since  $\nu_{>k}(\boldsymbol{\omega}) \in \Phi$, it follows that
		\begin{equation*}
			\begin{aligned}
				\nu_{>k}^{\mathbf{n}}(\boldsymbol{\omega})\Big(\dfrac{1}{2^{n_{k+1}-1}}[-M,M]\Big ) &\geq \nu_{>k}^{\mathbf{n}}(\boldsymbol{\omega})(N_{\omega_{k+1}}^{1-n_{k+1}} \cdot [-M, M]) \\
				&\geq \nu_{>k}(\boldsymbol{\omega})([-M, M]) \\
				&> 1 - \epsilon,
			\end{aligned}
		\end{equation*}
		that is to say,  the measure of $\nu_{>k}^{\mathbf{n}}(\boldsymbol{\omega})$ is concentrated on some neighborhood of the origin. Let  $E_{k} = \dfrac{1}{2^{n_{k+1}-1}}[-M,M]$. Then we have that
		$$
		\nu_{>k}^{\mathbf{n}}(E_k)>1-\epsilon.
		$$
		
Since $\vert e^{i\theta} - 1\vert \leq \vert \theta \vert$ for all $\theta \in [0,2\pi]$,  by \eqref{ineq_xiep}, we obtain that
		\begin{equation*}
			\int_{E_{k}} \vert e^{-2\pi i \xi x} - 1 \vert \mathrm{d}\nu_{>k}^{\mathbf{n}}(\boldsymbol{\omega}) \leq 2\pi \vert \xi \vert \int_{E_{k}} \vert x \vert \mathrm{d}\nu_{>k}^{\mathbf{n}}(\boldsymbol{\omega}) \leq \dfrac{M\pi \vert \xi \vert }{2^{n_{k+1}-2}}<\epsilon.
		\end{equation*}
		On the other hand, we have that
		\begin{equation*}
			\int_{\mathbb{R}/E_{k}} \vert e^{-2\pi i \xi x} - 1\vert \mathrm{d}\nu_{>k}^{\mathbf{n}}(\boldsymbol{\omega}) \leq 2\nu_{>k}^{\mathbf{n}}(\boldsymbol{\omega})(\mathbb{R}/E_{k}) \leq 2\epsilon.
		\end{equation*}
		Combining these together, we obtain that
		\begin{equation*}
			\vert \widehat{\nu_{>k}^{\mathbf{n}}(\boldsymbol{\omega})}(\xi) - 1 \vert\le \int_{\mathbb{R}} \vert e^{-2\pi i \xi x} - 1 \vert \mathrm{d}\nu_{>k}^{\mathbf{n}}(\boldsymbol{\omega})< 3\epsilon,
		\end{equation*}
		for all $k\geq K$, which implies that
		\begin{equation*}
			\lim\limits_{k\to\infty} \widehat{\nu_{>k}^{\mathbf{n}}(\boldsymbol{\omega})}(\xi) =1= \widehat{\delta_{0}}(\xi),
		\end{equation*}
		for all $\xi\in \R$.

		Therefore, by Lemma \ref{lemma_weak_convergence_equivalent_statement},  $\{\nu_{>k}^{\mathbf{n}}(\boldsymbol{\omega})\}$ converges weakly to $\delta_{0}$. Since  $\mathcal{Z}(\delta_{0}) = \emptyset$, it follows from Theorem $\ref{thm_ipzs}$ that  $\mu^{\mathbf{n}}(\boldsymbol{\omega})$ is a spectral measure for every $\boldsymbol{\omega}\in \Omega$, which completes the proof.
	\end{proof}

If $\mathbf{n}$ is bounded, we are only able to show that $\mu^\mathbf{n}$ is spectral $\mathbb{P}$-a.e.. To this end, we need  the following two conclusions.

	\begin{lemma} \label{lemma_gcd_of_B}
		Let $\{B_{j}\}_{j=1}^{\infty}$ be a sequence such that  $B_{j} \subseteq \mathbb{Z}$ and $\# B_j \geq 2$ for all $j\geq 1$. Then exactly one of the followings occurs:
		
		$(\rmnum{1})$ $\gcd \big( \bigcup_{j=1}^{\infty} B_{j} \big) > 1;$
		
		$(\rmnum{2})$ there exists a finite subset $\mathcal{I} \subseteq \mathbb{N}_{+}$ such that
		\begin{equation*}
			\gcd \Big( \bigcup_{i \in \mathcal{I}} B_{i} \Big) = 1.
		\end{equation*}
	\end{lemma}
	\begin{proof}
		It suffices to prove that the statement $(\rmnum{2})$ is true if and only if
		\begin{equation}\label{sfgcd1}
			\gcd\Big( \bigcup_{j=1}^{\infty} B_{j} \Big) = 1.
		\end{equation}
The necessity is obvious, and  we prove  the sufficiency by contradiction.

Assume that the statement $(\rmnum{2})$ does not hold, that is, 	for all finite subset $\mathcal{I}\subset \mathbb{N}_{+}$,
		\begin{equation*}
			\gcd \Big( \bigcup_{i \in \mathcal{I}} B_{i} \Big) \ge 2.
		\end{equation*}
		We write
		\begin{equation} \label{equation_gcd_partially_ordered_set}
			D = \{\gcd \big( \bigcup_{i \in \mathcal{I}} B_{i} \big): \mathcal{I} \subseteq \mathbb{N}_{+} \text{ is finite}\}.
		\end{equation}
		It is clear that  $D$ is  well-ordered with respect to $\leq$, and $D$ contains the least element $d\ge2$. Without loss of generality, we assume that a finite subset $\mathcal{I}=\{1,2,\ldots, n\} \subseteq \mathbb{N}_{+}$ such that
$$
d = \gcd \Big( \bigcup_{i=1}^n B_{i} \Big)\ge2.
$$

If there exists a positive integer $j_{0}\geq n+1$ such that $d \nmid \gcd(B_{j_{0}})$, then it immediately follows that
		\begin{equation*}
				\gcd\Big( \Big(\bigcup_{j=1}^{n} B_{j}\Big) \bigcup B_{j_{0}}\Big) < d,
		\end{equation*}
		which contradicts the fact that $d$ is the least element of $D$. Otherwise, if $d\mid \gcd(B_j)$ for all $j>n+1$, then we have $d\mid\gcd\big( \bigcup_{j=1}^{\infty} B_{j} \big)$ which contradicts \eqref{sfgcd1}.
	\end{proof}

	Let $\sigma$ denote the left shift on the symbolic space $\Omega$, that is,
	\begin{equation*}
		\sigma(\boldsymbol{\omega}) = \omega_{2}\omega_{3}\cdots\omega_{k}\cdots
	\end{equation*}
	for $\boldsymbol{\omega} =\omega_{1}\omega_{2}\cdots\omega_{k}\cdots \in \Omega$.

		\begin{lemma}\label{lemma_convergence_of_word}
			Let $\mathbb{P}$ be the Bernoulli measure on $\Omega$ given by \eqref{def_BPM} with respect to a positive  probability vector $\bp$.  Given $\boldsymbol{\alpha} \in \Omega$, there exists $\Omega_{0}\subset \Omega $ with $\mathbb{P}(\Omega_0)=1$ such that for each $\boldsymbol{\omega} \in \Omega_{0}$,  we have that
$$
\lim_{j\to \infty } \sigma^{k_{j}}(\boldsymbol{\omega}) =\boldsymbol{\alpha},
$$
for some  strictly increasing  sequence $\{k_{j}\}_{j=1}^{\infty}$.
		\end{lemma}

	\begin{proof}
		Since $\boldsymbol{\alpha} \in \Omega$ is given, we write $\boldsymbol{\alpha}=\alpha_1\alpha_2\alpha_3\cdots$.
		For each integer  $q\geq1$, choose a sequence $\{ k_{j}^{(q)}\}_{j=1}^{\infty}$ of positive integers  such that $k_{1}^{(q)} = 1$ and $k_{i+1}^{(q)} - k_{i}^{(q)} > q$ for all   $j>1$.

Fix $q$, for each integer $i\geq 1$,  we define a random variable $X_i^{(q)}: \Omega \to \R$ by
		\begin{equation*}
			X_{i}^{(q)}(\boldsymbol{\omega})=
			\begin{cases}
				1, & \omega_{k_{i}^{(q)}} \omega_{k_{i}^{(q)}+1} \dots \omega_{k_{i}^{(q)}+q-1} = \alpha_{1} \alpha_{2} \cdots \alpha_{q}, \\
				0, & \text{otherwise}.
			\end{cases}
		\end{equation*}
Since the Bernoulli measure  $\mathbb{P}$  is generated by the probability vector $\bp=(p_1,p_2,\ldots)$ where $p_j>0$ for all $j>0$, the expectation of $X_i^{(q)}$ is given by
		\begin{equation*}
			\mathbb{E}[X_{i}^{(q)}] = \mathbb{P}(X_{i}^{(q)} = 1) = p_{\alpha_{1}} p_{\alpha_{2}} \dots p_{\alpha_{q}}>0,
		\end{equation*}
		for all $i\geq1$. Hence $\{ X_{i}^{(q)}(\boldsymbol{\omega}) \}_{i=1}^{\infty}$ is a sequence of independently identically distributed integrable random variables.

By the  Kolmogorov strong law of large numbers, there exists a subset $\Omega_{q} \subseteq \Omega$ with $\mathbb{P}(\Omega_q)=1$ such that for all $\boldsymbol{\omega} \in \Omega_{q}$,
		\begin{equation*}
			\lim\limits_{n\to\infty} \dfrac{1}{n} \sum\limits_{i=1}^{n} X_{i}^{(q)}(\boldsymbol{\omega}) = E[X_{i}^{(q)}] = p_{\alpha_{1}} p_{\alpha_{2}} \dots p_{\alpha_{q}},
		\end{equation*}
which is equivalent to
		\begin{equation}\label{equation_Kolmogorov_strong_law_of_large_number_in_q_situation}
			\lim\limits_{n\to\infty} \dfrac{\# \{ 1\leq i \leq n : \omega_{k_{i}^{(q)}} \omega_{k_{i}^{(q)}+1} \dots \omega_{k_{i}^{(q)}+q-1} = \alpha_{1} \alpha_{2} \dots \alpha_{q} \}}{n} = p_{\alpha_{1}} p_{\alpha_{2}} \dots p_{\alpha_{q}}.
		\end{equation}

Since $\mathbb{P}(\Omega_q)=1$ for all integers $q\geq 1$, we write
		\begin{equation*}
			\Omega_{0} = \bigcap_{q=1}^{\infty} \Omega_{q},
		\end{equation*}
and it is clear that   $\mathbb{P}(\Omega_{0}) = 1.$  Hence for all $\boldsymbol{\omega} \in \Omega_{0}$ and all integers $q\geq 1$, we have that
		\begin{equation}\label{equation_Kolmogorov_strong_law_of_large_number}
			\lim\limits_{n\to\infty} \dfrac{\# \{ 1\leq i \leq n : \omega_{k_{i}^{(q)}} \omega_{k_{i}^{(q)}+1} \dots \omega_{k_{i}^{(q)}+q-1} = \alpha_{1} \alpha_{2} \dots \alpha_{q} \}}{n} = p_{\alpha_{1}} p_{\alpha_{2}} \dots p_{\alpha_{q}}>0.
		\end{equation}

For each given  $\boldsymbol{\omega} =\omega_1\omega_2\cdots\omega_n\cdots\in \Omega_{0}$, we define a sequence of integers $k_j$ inductively. First, for $q=1$, by \eqref{equation_Kolmogorov_strong_law_of_large_number}, there exists a sufficiently large integer $K_{1}\geq 1$ such that
$$
\# \{ 1\leq i \leq K_1 : \omega_{k_{i}^{(1)}} = \alpha_{1}\}>\frac{1}{2}K_1 \cdot p_{\alpha_{1}}>1.
$$
We choose $i_1\in \{ 1\leq i \leq K_1 : \omega_{k_{i}^{(1)}} = \alpha_{1}\}$ and have $\omega_{k_{i_1}^{(1)}} = \alpha_{1}$. By setting $k_1=k_{i_1}^{(1)}-1$, we have that
$$
\sigma^{k_{1}}(\boldsymbol{\omega}) =\alpha_1\omega_{k_1+2}\omega_{k_1+3}\cdots.
$$

Assume that the integers $\{ k_{j} \}_{j=1}^{l}$ have been chosen and  satisfy that   $k_{1} < k_{2} < \dots < k_{l}$ and
\begin{equation}\label{shifpro}
\sigma^{k_{j}}(\boldsymbol{\omega}) =\alpha_1\cdots\alpha_{j}\omega_{k_{j}+j+1}\omega_{k_{j}+j+2}\cdots.
\end{equation}
for all $1\leq j \leq l$. For  $q = l+1$, letting $\epsilon = \frac{1}{2}p_{\alpha_{1}} p_{\alpha_{2}} \dots p_{\alpha_{l+1}}$,  by \eqref{equation_Kolmogorov_strong_law_of_large_number}, there exists an  integer  $K_{l+1}\geq 1$ such that 	
		$$
\# \{1\leq i \leq n: \omega_{k_{i}^{(l+1)}} \omega_{k_{i}^{(l+1)}+1} \dots \omega_{k_{i}^{(l+1)}+l} = \alpha_{1} \alpha_{2} \dots \alpha_{l+1} \} \geq n (p_{\alpha_{1}} p_{\alpha_{2}} \dots p_{\alpha_{l+1}} - \epsilon).
		$$
for all $n>K_{l+1}$. For a sufficiently  large $n$, we choose $i_{l+1}$ such that $1\le i_{l+1}\le n$, $k_{i_{l+1}}^{(l+1)} > k_{{l}}+1$  and
		\begin{equation*}
			\omega_{k_{i_{l+1}}^{(l+1)}} \omega_{k_{i_{l+1}}^{(l+1)}+1} \dots \omega_{k_{i_{l+1}}^{(l+1)}+l} = \alpha_{1} \alpha_{2} \dots \alpha_{l+1}.
		\end{equation*}
	 Setting  $k_{l+1} = k_{i_{l+1}}^{(l+1)}-1$,  we have that
$$
\sigma^{k_{l+1}}(\boldsymbol{\omega}) =\alpha_1\cdots\alpha_{l+1}\omega_{k_{l+1}+l}\omega_{k_{l+1}+l+1}\cdots.
$$

Hence for each given $\boldsymbol{\omega}$, we obtain a sequence $\{k_j\}_{j=1}^\infty$ satisfying \eqref{shifpro}, and it follows  that
$$
d(\sigma^{k_{j}}(\boldsymbol{\omega}),\bu) \leq  2^{-j}
$$
for all $j\geq 1$. Therefore  $\{ \sigma^{k_{j}} (\boldsymbol{\omega}) \}$ converges to $\boldsymbol{\alpha} \in \Omega$,  and the conclusion holds.

	\end{proof}
	
	Finally, we are ready to prove that the mapping $\mu^{\mathbf{n}}$ is a spectral random  measure for $\mathbb{P}$-a.e. $\boldsymbol{\omega} \in \Omega$.
	
\begin{proof}[Proof of Theorem \ref{thm_TC}]
Fix a positive integral sequence $\bn=\{n_k\}_{k=1}^\infty$. We only consider that  $\mathbf{n}$ is bounded since  the conclusion follows from  Theorem \ref{thm_ub} if  $\mathbf{n}$ is unbounded.  Fix a Bernoulli measure $\mathbb{P}$ on $\Omega$ associated with a probability vector $\bp=(p_{1}, p_{2}, \dots)$. 
Without loss of generality, we assume that $p_{k} \geq p_{k+1} > 0$ for all $k\geq 1$.
		
For the given sequence $\{(N_{k}, B_{k})\}_{k=1}^{\infty}$ where  $B_{k} \subseteq \N$ and $\# B_k \geq 2$ for all $k\geq 1$, by Lemma~\ref{lemma_gcd_of_B}, the proof is divided into the following two cases: $(\rmnum{1})$ there exists a finite subset $\mathcal{I} \subseteq \mathbb{N}_{+}$ such that $	\gcd \big( \bigcup_{j \in \mathcal{I}} B_{j} \big) = 1$; $(\rmnum{2})$ $\gcd \big( \bigcup_{j=1}^{\infty} B_{j} \big) > 1$.

 $(\rmnum{1})$:  We assume that  $\mathcal{I} = \{j_{1}, j_{2}, \dots, j_{m}\} \subseteq \mathbb{N}$  such that  $\gcd \big( \bigcup_{j \in \mathcal{I}} B_{j} \big) = 1.$ Since $0\in B_k$ for all $k>0$, we have
 \begin{equation} \label{equation_gcd_B-B=1}
\gcd \bigg( \bigcup_{j \in \mathcal{I}} (B_{j} - B_{j}) \bigg) = 1.
\end{equation}		
Setting  $\boldsymbol{\alpha} = \alpha_{1} \alpha_{2}\cdots = (j_{1} j_{2}\dots j_{m})^{\infty}\in \Omega$. By Lemma $\ref{lemma_convergence_of_word}$, there exists a full measure subset $\Omega_{0} \subseteq \Omega$ such that for every $\boldsymbol{\omega}\in\Omega_{0}$, there exists  a strictly  increasing   sequence $\{k_{j}\}_{j=1}^{\infty}$ satisfying that
$$
\lim_{j\to \infty } \sigma^{k_{j}}(\boldsymbol{\omega}) =\boldsymbol{\alpha}.
$$

Since $\mathbf{n}$ is bounded, we write  $M = \sup\{n_{k}: k\geq 1\}$ and  $\Sigma = \{1, 2, \dots, M\}^{\mathbb{N}}$ for the compact symbolic space over the alphabet $\{1,2,\dots, M\}$.  Let $\nu_{>k}^{\mathbf{n}}(\boldsymbol{\omega})$ and  $\nu_{>k}(\boldsymbol{\omega})$ be given  respectively by  \eqref{def_nu_words} and \eqref{def_nu1geqk} with respect to  $\mu^{\mathbf{n}}(\boldsymbol{\omega})$. We write  $\mathbf{n}^{(k)} = \{ n_{k+l} \}_{l=1}^{\infty}$, and it is clear that for each $k\geq 1$,
\begin{equation*}
		\nu_{>k}^{\mathbf{n}}(\boldsymbol{\omega}) = \mu^{\mathbf{n}^{(k)}}(\sigma^{k}(\boldsymbol{\omega})).
\end{equation*}

Fix $\boldsymbol{\omega} \in \Omega_{0}$,  and let $\{k_{j}\}_{j=1}^{\infty}$ be the strictly increasing sequence  such that
$$
\lim_{j\to \infty } \sigma^{k_{j}}(\boldsymbol{\omega}) =\boldsymbol{\alpha}.
$$
 Since  $\Sigma$ is compact, the sequence
			\begin{equation*}
				\{ (\sigma^{k_{j}}(\boldsymbol{\omega}), \{ n_{k_{j}+l} \}_{l=1}^{\infty}) \}_{j=1}^{\infty}
			\end{equation*}			
has a convergent subsequence in $\Omega \times \Sigma$. Without loss of generality, we assume that $\{ (\sigma_{k_{j}}(\boldsymbol{\omega}), \{ n_{k_{j}+l} \}_{l=1}^{\infty}) \}_{j=1}^{\infty}$ converges to $(\boldsymbol{\alpha},\mathbf{m})$ for some sequence $\mathbf{m} = \{m_{k}\} \in \Sigma$.

By the same argument in the proof of Lemma \ref{lemma_continuous_mapping_from_symbol_space_to_Prob_space}, we have  that $\{ \nu_{>k_j}^{\mathbf{n}}(\boldsymbol{\omega}) \}_{j=1}^{\infty}$ weakly converges to $\mu^{\mathbf{m}}(\boldsymbol{\alpha})$.  By Theorem $\ref{thm_IPZSE}$, we have that $\mathcal{Z}(\mu^{\mathbf{m}}(\boldsymbol{\alpha})) = \emptyset$. Therefore,  by Theorem $\ref{thm_ipzs}$, we have that $\mu^{\mathbf{n}}(\boldsymbol{\omega})$ is a spectral measure for  every $\boldsymbol{\omega} \in \Omega_{0}$. Since $\mathbb{P}(\Omega_0)=1$, the  conclusion holds.

 $(\rmnum{2})$: We assume   $\gcd\big( \bigcup_{j=1}^{\infty} B_{j}  \big) = d > 1$. For each $j\geq 1$, we write $B_{j}' = B_{j} /d$, and $(N_{j}, B_{j}')$ is also an admissible pair since $\gcd\big( \bigcup_{j=1}^{\infty} B_{j}'\big) = 1$. It implies that there exists a finite subset $\mathcal{I}\subseteq \mathbb{N}$  such that  $\gcd \big( \bigcup_{j \in \mathcal{I}} B_{j}' \big) = 1.$

For every  $\boldsymbol{\omega} \in \Omega$,  let	
\begin{equation*}
			(\mu^{\mathbf{n}})'(\boldsymbol{\omega}) = \delta_{N_{\omega_{1}}^{-n_{1}}B_{\omega_{1}}'} * \delta_{N_{\omega_{1}}^{-n_{1}}N_{\omega_{2}}^{-n_{2}}B_{\omega_{2}}'} * \dots * \delta_{N_{\omega_{1}}^{-n_{1}}N_{\omega_{2}}^{-n_{2}}\dots N_{\omega_{k}}^{-n_{k}}B_{\omega_{k}}'} * \cdots .
		\end{equation*}
be the infinite convolution associated with  $\{(N_{j}, B_{j}')\}_{j=1}^{\infty}$ and a fixed $\mathbf{n}$. By   $(\rmnum{1})$,  we have that $(\mu^{\mathbf{n}})'(\boldsymbol{\omega})$ is a spectral measure for $\mathbb{P}$-a.e. $\boldsymbol{\omega} \in \Omega$. Setting  $T_{d, 0}(x)= d x+0$, we  have that
\begin{equation*}
			\mu^{\mathbf{n}}(\boldsymbol{\omega}) = (\mu^{\mathbf{n}})'(\boldsymbol{\omega}) \circ T_{d,0}^{-1}.
		\end{equation*}
By Lemma \ref{lemma_invariant_of_spectrality_under_lineartransf}, we obtain that $\mu^{\mathbf{n}}(\boldsymbol{\omega})$ is a spectral measure for $\mathbb{P}$-a.e. $\boldsymbol{\omega} \in \Omega$.
	\end{proof}

	\section{Existence and tightness of spectral random measures}\label{sec_ex}
	
In this section, we provide some sufficient conditions for  the existence and tightness of random measures.
Recall that 	a sequence $\{( N_k,B_k)\}_{k=1}^\infty $ satisfies  remainder bounded condition (RBC) if
\begin{equation}\label{cdn_RBC}
\sum_{k=1}^{\infty} \frac{\# B_{k,2}}{\# B_k} < \infty,
\end{equation}
		where $B_{k,1}=B_k \cap \{0,1,\cdots,N_k-1\}$ and $B_{k,2}=B_k \backslash B_{k,1}.$

	\begin{proof}[Proof of Theorem \ref{thm_tight}]
Fix  $\ba=\alpha_1\alpha_2\cdots \in \Omega$. Recall that $\Omega=\mathbb{N}^{\mathbb{N}}$,  and we obtain  $\{( N_{\alpha_k},B_{\alpha_k})\}_{k=1}^\infty $ from the given sequence $\{( N_k,B_k)\}_{k=1}^\infty $. To show the existence of the infinite convolution $\mu(\ba)$, by Theorem \ref{thm_weak-convergence-infinite-convolution}, it is equivalent to prove that
		\begin{equation}\label{eq_existence}
			\sum_{k=1}^{\f} \dfrac{1}{\# B_{\alpha_k}} \sum_{b \in B_{\alpha_{k}}} \dfrac{b}{N_{\alpha_{1}}N_{\alpha_{2}}\cdots N_{\alpha_{k}}+b}<\infty.
		\end{equation}

First, we define a set $T=\{t_1,t_2,\ldots,\}$		inductively.  Let	$t_{1}= 1$, and we write  $t_2=\inf\{n:\alpha_n\ne\alpha_{t_1}\}$ if $\inf\{n:\alpha_n\ne\alpha_{t_1}\}<\infty$ otherwise we write $T=\{t_1\}$.  Suppose that $t_{k-1}$ has been defined. Then we write $t_k=\inf\{n:\alpha_n\ne\alpha_{t_m}, \text{for all} \ 1\le m<k\}$ if $\inf\{n:\alpha_n\ne\alpha_{t_m}, \text{for all} \ 1\le m<k\}<\infty$ otherwise we set $T=\{t_1,t_2,\ldots, t_{k-1}\}$. Then we obtain a sequence of integers $t_k$, and let $T$ be the collection of all these integers $t_k$.

Letting  $\mathbf{N}_{\alpha_{t_{k-1}}}=N_{\alpha_{t_1}}N_{\alpha_{t_2}}\cdots N_{\alpha_{t_{k-1}}}$,  we have that
		\begin{eqnarray}\label{eq_existence_rewritten_form}
		\sum_{k=1}^{\f} \dfrac{1}{\# B_{\alpha_k}} \sum_{b \in B_{\alpha_{k}}} \dfrac{b}{N_{\alpha_{1}}N_{\alpha_{2}}\cdots N_{\alpha_{k}}+b} &=& \sum_{k=1}^{\# T} \sum_{\{l:\alpha_l=\alpha_{t_k}\}} \dfrac{1}{\# B_{\alpha_{t_{k}}}} \sum_{b \in B_{\alpha_{t_{k}}}} \dfrac{b}{N_{\alpha_{1}}N_{\alpha_{2}}\dots N_{\alpha_{l}} + b}  \nonumber \\
&\le&	\sum_{k=1}^{\# T} \sum_{j=1}^{\# \{l:\alpha_l=\alpha_{t_k}\}} \dfrac{1}{\# B_{\alpha_{t_{k}}}} \sum_{b \in B_{\alpha_{t_{k}}}} \dfrac{b}{\mathbf{N}_{\alpha_{t_{k-1}}}N_{\alpha_{t_k}}^j + b}.   \label{eq_existence_rewritten_form_1}
		\end{eqnarray}

Next,  for each fixed $k\in\{1,2\cdots,\# T\}\cap\N$, we estimate the terms in \eqref{eq_existence_rewritten_form_1} separately.
Since $\sup_{k\ge 1}\frac{\log\max B_k}{\log N_k}<\f$, there exists a positive integer $C>0$ such that 		
		\begin{equation}\label{eq_Bk<NkC}
			\max \{b:b\in B_k\}<N_k^C,
		\end{equation}
		for all $k\ge1$.

	If $\# \{l:\alpha_l=\alpha_{t_k}\}\le C$, by writing  $B_{\alpha_k,1}=B_{\alpha_k} \cap \{0,1,\cdots,N_{\alpha_k}-1\}$ and $B_{\alpha_k,2}=B_{\alpha_k} \backslash B_{\alpha_k,1}$, we have that
	
		\begin{eqnarray*} 
&& \hspace{-1cm}\sum_{j=1}^{\# \{l:\alpha_l=\alpha_{t_k}\}} \dfrac{1}{\# B_{\alpha_{t_{k}}}} \sum_{b \in B_{\alpha_{t_{k}}}} \dfrac{b}{\mathbf{N}_{\alpha_{t_{k-1}}}N_{\alpha_{t_k}}^j + b} \leq \sum_{j=1}^{C} \dfrac{1}{\# B_{\alpha_{t_{k}}}} \sum_{b \in B_{\alpha_{t_{k}}}} \dfrac{b}{\mathbf{N}_{\alpha_{t_{k-1}}}N_{\alpha_{t_k}}^j + b}\\
&&=\sum_{j=1}^{C} \dfrac{1}{\# B_{\alpha_{t_{k}}}} \sum_{b \in B_{\alpha_{t_{k}},1}} \dfrac{b}{\mathbf{N}_{\alpha_{t_{k-1}}}N_{\alpha_{t_k}}^j + b}+\sum_{j=1}^{C} \dfrac{1}{\# B_{\alpha_{t_{k}}}} \sum_{b \in B_{\alpha_{t_{k}},2}} \dfrac{b}{\mathbf{N}_{\alpha_{t_{k-1}}}N_{\alpha_{t_k}}^j + b}\\
&&\le\frac{1}{\mathbf{N}_{\alpha_{t_{k-1}}}}\sum_{j=1}^{C} \dfrac{1}{\# B_{\alpha_{t_{k}}}} \sum_{b \in B_{\alpha_{t_{k}},1}} \dfrac{b}{N_{\alpha_{t_k}}^j}+\sum_{j=1}^{C} \dfrac{\# B_{\alpha_{t_{k}},2}}{\# B_{\alpha_{t_{k}}}}\\
&&\le\frac{1}{\mathbf{N}_{\alpha_{t_{k-1}}}}\sum_{j=1}^{C} \dfrac{N_{\alpha_{t_k}}-1}{N_{\alpha_{t_k}}^j}+C\cdot \dfrac{\# B_{\alpha_{t_{k}},2}}{\# B_{\alpha_{t_{k}}}}\\
&&\le\frac{1}{\mathbf{N}_{\alpha_{t_{k-1}}}}+C\cdot \dfrac{\# B_{\alpha_{t_{k}},2}}{\# B_{\alpha_{t_{k}}}}.
		\end{eqnarray*}

If $\# \{l:\alpha_l=\alpha_{t_k}\}>C$, by $\eqref{eq_Bk<NkC}$,  we have
$$
		\sum_{j=1}^{C} \dfrac{1}{\# B_{\alpha_{t_{k}}}} \sum_{b \in B_{\alpha_{t_{k}}}} \dfrac{b}{\mathbf{N}_{\alpha_{t_{k-1}}}N_{\alpha_{t_k}}^j + b}\le\frac{1}{\mathbf{N}_{\alpha_{t_{k-1}}}}+C\cdot \dfrac{\# B_{\alpha_{t_{k}},2}}{\# B_{\alpha_{t_{k}}}},
$$		
and 	
		\begin{equation*}
			\begin{aligned}
				\sum_{j=C+1}^{\# \{l:\alpha_l=\alpha_{t_k}\}} \dfrac{1}{\# B_{\alpha_{t_{k}}}} \sum_{b \in B_{\alpha_{t_{k}}}} \dfrac{b}{\mathbf{N}_{\alpha_{t_{k-1}}}N_{\alpha_{t_k}}^j + b}
				&\le\frac{1}{\mathbf{N}_{\alpha_{t_{k-1}}}}\sum_{j=C+1}^{\# \{l:\alpha_l=\alpha_{t_k}\}} \dfrac{1}{\# B_{\alpha_{t_{k}}}} \sum_{b \in B_{\alpha_{t_{k}}}} \dfrac{b}{N_{\alpha_{t_k}}^j}\\
				&\le\frac{1}{\mathbf{N}_{\alpha_{t_{k-1}}}}\sum_{j=C+1}^{\# \{l:\alpha_l=\alpha_{t_k}\}}\dfrac{\max\{b:b\in B_{\alpha_{t_k}}\}}{N_{\alpha_{t_k}}^j}\\
				&\le\frac{1}{\mathbf{N}_{\alpha_{t_{k-1}}}}.
			\end{aligned}
		\end{equation*}

Since $\{( N_k,B_k)\}_{k=1}^\infty $ satisfies RBC, by\eqref{cdn_RBC}, we have
		\begin{eqnarray*}
				\sum_{k=1}^{\f} \dfrac{1}{\# B_{\alpha_k}} \sum_{b \in B_{\alpha_{k}}} \frac{b}{N_{\alpha_{1}}N_{\alpha_{2}}\cdots N_{\alpha_{k}}+b} &\le&\sum_{k=1}^{\# T}\bigg(\frac{2}{\mathbf{N}_{\alpha_{t_{k-1}}}}+C\cdot \dfrac{\# B_{\alpha_{t_{k}},2}}{\# B_{\alpha_{t_{k}}}}\bigg)< \f,
		\end{eqnarray*}
 and we obtain the inequality \eqref{eq_existence}, which implies that infinite convolution $\mu^{}(\ba)$ exists for all $\ba\in\Omega$.

		Next we  prove  $\Phi$ is tight.  For each $k\in\mathbb{N}$, we write
		\begin{eqnarray} \label{def_phi_k}
&& \Phi_{k} = \{ \delta_{(N_{k}M_{1})^{-1}B_{k}} * \delta_{(N_{k}^{2}M_{1}M_{2})^{-1}B_{k}} * \dots  * \delta_{(N_{k}^{j}M_{1}M_{2} \dots M_{j})^{-1}B_{k}} * \dots : \\
 &&\hspace{6cm}  M_{j} \in \mathbb{N}\cup \{\infty\} \text{ for all } j\geq 1  \},                \nonumber
		\end{eqnarray}
where $\delta_{(N_{k}^{j}M_{1}M_{2} \dots M_{j})^{-1}B_{k}}=\delta_{0}$ if $M_{1}M_{2} \dots M_{j}=\infty$.

Arbitrarily choose an element from $\Phi_k$, 	and we write it as
\begin{equation}\label{def_etaM}
\eta_{\mathbf{M}} = \delta_{(N_{k}M_{1})^{-1}B_{k}} * \delta_{(N_{k}^{2}M_{1}M_{2})^{-1}B_{k}} * \dots * \delta_{(N_{k}^{j}M_{1}M_{2} \dots M_{j})^{-1}B_{k}} * \dots
\end{equation}
to emphasized the dependence on the sequence $\mathbf{M}=\{M_k\}_{k=1}^\infty$. By \eqref{eq_Bk<NkC}, for each element $x$ contained in the set
		\begin{equation*}
			(N_{k}M_{1})^{-1}B_{k, 1} + \dots +  (N_{k}^{C}M_{1}M_{2} \dots M_{C})^{-1}B_{k,1} + (N_{k}^{C+1}M_{1}M_{2} \dots M_{C+1})^{-1}B_{k} + \cdots,
		\end{equation*}
		we have
		\begin{equation*}
			0\le x \leq \dfrac{1}{M_{1}} \cdot \bigg( \sum_{j=1}^{C} \dfrac{N_{k}-1}{N_{k}^{j}} + \sum_{j=C+1}^{\infty} \dfrac{\max B_{k}}{N_{k}^{j}} \bigg) < \dfrac{2}{M_{1}},
		\end{equation*}
		which implies
		\begin{equation*}
			\eta_{\mathbf{M}}\Big(\big[0,\dfrac{2}{M_{1}}\big)\Big) \geq \bigg( \dfrac{\# B_{k,1}}{\# B_{k}} \bigg)^{C} \geq 1 - C\dfrac{\# B_{k,2}}{\# B_{k}},
		\end{equation*}
		for all $\eta_{\mathbf{M}} \in \Phi_{k}$, where the second inequality follows from Bernoulli's inequality. Therefore,
		\begin{equation*}
			\eta_{\mathbf{M}}\Big(\big[\dfrac{2}{M_{1}}, \infty\big)\Big) = 1 - \eta_{\mathbf{M}}\Big(\big[0,\dfrac{2}{M_{1}}\big)\Big)  \leq C\dfrac{\# B_{k,2}}{\# B_{k}},
		\end{equation*}
		for all $\eta_{\mathbf{M}} \in \Phi_{k}$.
		
		Since $\{(N_{k}, B_{k})\}_{k=1}^{\infty}$ satisfies RBC, for each $\epsilon>0$, there exists $n_0>2$ such that
		\begin{equation*}
			\sum_{k=n_0}^{\infty} \dfrac{\# B_{k,2}}{\# B_{k}} < \dfrac{\epsilon}{C},
		\end{equation*}
and we write
		\begin{equation*}
			\begin{aligned}
				\Phi_{<n_0}=&\{\delta_{(N_{m_{1}}M_{1})^{-1}B_{m_{1}}} * \dots * \delta_{(N_{m_{1}}N_{m_{2}}\dots N_{m_{j}}M_{1}M_{2} \dots M_{j})^{-1}B_{m_{j}}} * \dots :\\
				& \hspace{3cm} M_{j} \in \mathbb{N}\cup \{\infty\}, m_{j} \in [1, n_0-1] \cap \mathbb{N} \text{ for all } j\in \mathbb{N} \}.
			\end{aligned}
		\end{equation*}
 Since $\Phi_{<n_0}$ only contains the infinite convolutions generated by finite admissible pairs $\{(N_{k}, B_{k})\}_{k=1}^{n_0-1}$, there exists a compact subset $K$ such that $spt(\eta) \subseteq K$ for all $\eta\in \Phi_{<n_0}$.

		By Lemma \ref{lem_factorization}, for each $\mu\in\Phi$, we may rearrange the order of the convolution,
		$$
		\mu=\mu_{<n_0}*\mu_{s_1}*\mu_{s_2}*\cdots*\mu_{s_k}*\cdots,
		$$
		where $\mu_{<n_0}\in\Phi_{< n_0}$, $\mu_{s_k}\in\Phi_{s_k}$, $s_k\ne s_j$ for all $k\ne j$. Recall \eqref{def_phi_k}  and \eqref{def_etaM}, it is clear that for the sequence $\mathbf{M}=\{M_k\}_{k=1}^\infty$ associated with $\mu_{s_k}\in\Phi_{s_k}$, we have  that  $N_{s_1}N_{s_2}\cdots N_{s_{k-1}}|M_1$. Hence, we have $spt(\mu_{<n_0})\subseteq K$, $\mu_{s_1}\big([2,+\f)\big)\le \frac{C \# B_{s_1,2}}{B_{s_1}}$ and
\begin{equation}\label{eq_spt}
\mu_{s_k}\Big(\big[\frac{2}{N_{s_1}N_{s_2}\cdots N_{s_{k-1}}},+\f\big)\Big)\le \frac{C \# B_{s_k,2}}{B_{s_k}},
\end{equation}		
for all $k\ge2$. We write $K'=K+[0,4]=\{a+c:a\in K,c\in[0,4]\}$ and it is clear that $K'$ is compact. By \eqref{eq_spt}, we have
		$$
		\mu(\R\setminus K')\le\sum_{k=n_0}^{\f}\frac{C \# B_{k,2}}{B_k}< \epsilon.
		$$
		
		Therefore, we obtain that $\Phi = \{\mu_{\ba} :\ba \in \Omega\}$ is tight.
	\end{proof}

		\begin{proof}[Proof of Corollary $\ref{thm_CC}$]
			Since $	\sup\limits_{n\geq1}\{N_{n}^{-1}b:b\in B_{n}\} < \infty$, we have that
			$$
			\sup_{k\ge 1}\frac{\log\max B_k}{\log N_k}<\f.
			$$ 			
			By Theorem \ref{thm_tight}, it implies that $\Phi$ is tight, and the conclusion immediately follows from Theorem \ref{thm_measurable} and Theorem \ref{thm_TC}.
		\end{proof}

Finally, we construct a spectral random measure without compact support, that is, there does not exist a compact set which contains the support of $\mu(\boldsymbol{\omega}) $ for all $\boldsymbol{\omega}\in \Omega$, but $\Phi$ is tight.

\begin{example}\label{ex_tnc}
			For each $k>0$,  write $N_{k} = 4^{2^{k-1}}$ and $B_{k} = \{0, 1, 2, \dots, 2^{2^{k-1}}-2, 4^{2^{k}}-1\}$, and it is clear that $\{(N_{k}, B_{k})\}_{k=1}^\infty $ is a sequence of  admissible pairs. Then $\Phi = \{ \mu(\boldsymbol{\omega})  : \boldsymbol{\omega} \in \Omega\}$ is tight. Moreover, for every positive integral sequence $\bn$ and every Bernoulli probability $\mathbb{P}$,  the mapping $\mu^\bn$ given by \eqref{def_random_measure} is a  spectral random measure for  $\mathbb{P}$-a.e. $\boldsymbol{\omega}$.

			Since
			$$
			\sum_{k=1}^{\f}\frac{\# B_{k,2}}{\# B_k}=\sum_{k=1}^{\f}\frac{1}{2^{2^{j-1}}}<\f,
			$$
			and
			$$
			\sup_{k\ge 1}\frac{\log \max B_k}{\log N_k}\le2,
			$$	
by Theorem \ref{thm_tight}, we have that $\mu(\boldsymbol{\omega})$ exists for  every  $ \boldsymbol{\omega} \in \Omega $ and  $\Phi = \{\mu(\omega) : \omega \in \Omega\}$ is tight. Hence,  by Theorem \ref{thm_TC}, $\mu^\bn$ is a  spectral random measure for  $\mathbb{P}$-a.e $\boldsymbol{\omega}$.

Let $\boldsymbol{\omega} = 123\dots \in \Omega$. 		Since
			\begin{equation*}
				\sum_{j=1}^{\infty} \dfrac{4^{2^{j}}-1}{N_{1} N_{2} \dots N_{j}} = \sum_{j=1}^{\infty} \dfrac{4^{2^{j}}-1}{4^{2^{j}-1}} = \infty,
			\end{equation*}
we have that  $spt(\mu(\boldsymbol{\omega}))$ is unbounded. 	Hence $spt(\mu(\boldsymbol{\omega}))$ is not contained in any compact subset of  $\mathbb{R}$.

	\end{example}

	\section{Intermediate-value Properties of Random Measures}\label{sec_IVP}
	
	In this section, we study the intermediate-value properties of spectral random measures, and we are able to construct various spectral random measures with rich geometric structure.
	
	For simplicity, we use finitely many admissible pairs $\{(N_{k}, B_{k})\}_{k=1}^{M}$ where $0\in B_k \subseteq \mathbb{N}$ for all $k>0$  and assume  $n_{k}=1$ for $k\geq 1$. We write $\Omega = \{1,\dots, M\}^{\mathbb{N}}$ for  the Corresponding symbolic space. Let $\mathbb{P}$ be the Bernoulli probability measure on $\Omega$ generated by  the probability vector $(p_{1},p_{2},\dots, p_{M})$. Let $\mu(\boldsymbol{\omega})$ be given by \eqref{def_corresponding_P_of_random_measure}. By Corollary \ref{cor_fad},  for every Bernoulli measure $\mathbb{P}$ on $\Omega$, $\mu(\boldsymbol{\omega})$ is a  spectral random measure for  $\mathbb{P}$-a.e. $\boldsymbol{\omega} \in \Omega$, and we have following conclusion on the dimensions.
	\begin{lemma}\label{lemma_dim_expression}
	Under the above assumptions. Hausdorff dimension and packing dimension of random measure $\mu$ are  given by
		\begin{equation}\label{equation_dim_expression}
			\dimh \mu=\dimp \mu= \dfrac{\sum_{k=1}^{M} p_{k} \log \# B_{k}}{\sum_{k=1}^{M} p_{k} \log N_{k}},
		\end{equation}
		for $\mathbb{P}$-a.e. $\boldsymbol{\omega} \in \Omega$.
	\end{lemma}
	\begin{proof}
		We only give the proof for Hausdorff dimension since the argument for packing dimension is almost identical.
		
		Fix $\boldsymbol{\omega} \in \Omega$, and we write $spt \big(\mu(\boldsymbol{\omega})\big)$ for the support of $\mu(\boldsymbol{\omega})$. Since $\mu(\boldsymbol{\omega})$ is uniformly distributed on its support, for every $x \in spt\big(\mu(\boldsymbol{\omega})\big)$, the local dimension of $\mu(\boldsymbol{\omega})$ at $x$ is given by
$$
				\underline{\dim}_{\mathrm{loc}}\mu(\boldsymbol{\omega})(x) = \varliminf_{r\to 0} \dfrac{\log \mu(B(x,r))}{\log r} = -\varliminf_{r\to 0} \dfrac{\log(\#B_{\omega_{1}}\#B_{\omega_{2}}\cdots \#B_{\omega_{k(r)}}) }{\log r},
$$
		where $k(r)$ is the positive integer such that
		\begin{equation}\label{eq_kr}
			\dfrac{1}{N_{\omega_{1}}N_{\omega_{2}}\cdots N_{\omega_{k(r)}}} \leq r < \dfrac{1}{N_{\omega_{1}}N_{\omega_{2}}\cdots N_{\omega_{k(r)-1}}}.
		\end{equation}		
		Since  $k(r)$ tends to infinity as $r\to 0$ and $\sup \{N_{k} : k = 1, \dots, M\} < \infty$,  the inequality \eqref{eq_kr} implies
		\begin{equation*}
			\lim\limits_{r\to 0} \dfrac{\log(N_{\omega_{1}}N_{\omega_{2}}\cdots N_{\omega_{k(r)}})}{\log r} = -1.
		\end{equation*}
		Hence we have
		\begin{equation*}
			\underline{\dim}_{\mathrm{loc}}\mu(\boldsymbol{\omega})(x) =\varliminf_{k\to \infty}\dfrac{\log(\#B_{\omega_{1}}\#B_{\omega_{2}}\cdots \#B_{\omega_{k}}) }{\log(N_{\omega_{1}}N_{\omega_{2}}\cdots N_{\omega_{k}})}= \varliminf_{k\to \infty}\dfrac{\sum_{i=1}^k \log(\#B_{\omega_{i}}) }{\sum_{i=1}^k \log N_{\omega_{i}}},
		\end{equation*}
		for all $x\in spt (\mu(\boldsymbol{\omega}))$. By Frostman Lemma \cite{Falcon97}, the Hausdorff dimension of $\mu(\boldsymbol{\omega})$ is
		\begin{equation}\label{eq_dim_H_associated_to_r}
			\dimh\mu(\boldsymbol{\omega}) = \varliminf_{k\to \infty}\dfrac{\sum_{i=1}^k \log\#B_{\omega_{i}} }{\sum_{i=1}^k \log N_{\omega_{i}}}.
		\end{equation}

		Since $\mathbb{P}$ is the  Bernoulli probability measure on $\Omega $ generated by $(p_{1},p_{2}, \dots, p_{M})$, $\mathbb{P}$ is an ergodic measure with respect to  the left shift $\sigma$ on $\Omega$, see \cite{Walters-1982} for details. By the Birkhoff's ergodic theorem,  for each $ k=1,2,\ldots,M$, we have that  for $\mathbb{P}$-a.e. $\bu=\alpha_1\alpha_2\cdots \in \Omega$,
		\begin{equation}\label{eq_frequence_and_probability_Birkhoff}
			\lim\limits_{n\to\infty}\dfrac{\# \{1\leq i \leq n : \alpha_{i} = k\}}{n} = p_{k}.
		\end{equation}
		Combining \eqref{eq_frequence_and_probability_Birkhoff}  with \eqref{eq_dim_H_associated_to_r}, we obtain that
		\begin{eqnarray*}
			\dimh\mu(\boldsymbol{\omega}) &=& \varliminf_{n\to \infty}\dfrac{ \sum_{i=1}^{M} \log\#B_{\omega_{i}} }{ \sum_{i=1}^{M}\log N_{\omega_{i}}}\\
			&=& \lim_{n\to \infty} \dfrac{ \frac{1}{n}\sum_{i=1}^{M} \#\{1\leq i \leq n : \omega_{i} = k\} \log\#B_{\omega_{i}} }{\frac{1}{n}\sum_{i=1}^{M} \#\{1\leq i \leq n : \omega_{i} = k\} \log N_{\omega_{i}}}\\
			&=&\dfrac{\sum_{k=1}^{M} p_{k} \log \# B_{k}}{\sum_{k=1}^{M} p_{k} \log N_{k}},
		\end{eqnarray*}
	for $\mathbb{P}$-a.e. $\boldsymbol{\omega} \in \Omega$,	and the conclusion holds.
	\end{proof}

	We write
	$$
	D=\{\mathbf{x} = (x_{1},x_{2},\dots, x_{M}) \in \mathbb{R}^{M} : \sum_{i=1}^M x_{i} = 1, 0\leq x_{i} \leq 1, i = 1,2,\dots, M \}.
	$$
Given $0<a_{i}\leq b_{i}, i=1,2,\dots, M$, we define a function $f:D\to (0,1]$ by
		\begin{equation*}
			f(\mathbf{x}) = \dfrac{\sum_{i=1}^{M}a_{i}x_{i}}{\sum_{i=1}^{M}b_{i}x_{i}},
		\end{equation*}
and we have the following simple fact.	
	\begin{proposition}\label{prop_dim_property}
		The function $f$ is continuous on $D$ with
		\begin{equation*}
			\max\limits_{\mathbf{x}\in D} f(\mathbf{x}) = \max\limits_{1\leq i \leq M} \frac{a_{i}}{b_{i}}
			\qquad	\textit{and } \qquad		\min\limits_{\mathbf{x}\in D} f(\mathbf{x}) = \min\limits_{1\leq i \leq M} \frac{a_{i}}{b_{i}}.
		\end{equation*}
	\end{proposition}

	Next, we prove that random measures have intermediate-value property by applying Lemma \ref{lemma_dim_expression} and Proposition \ref{prop_dim_property}.
	
	\begin{proof}[Proof of Theorem \ref{thm_imp}]
		For $s\in(0,1]$, there exists an integer  $n_{0}>0$ such that $\frac{1}{n_{0}}<s$. Given an integer $N\geq 2$, we set $N_{1} = N^{n_{0}}$ and $B_1=\{0,1,\ldots, N-1\}$,  and for each $k = 2, 3, \dots, M$, we write
		\begin{equation*}
			N_{k} = N,\qquad \textit{and } \qquad B_{k} = \{0,1,\dots, N-1\}.
		\end{equation*}
		It is clear that  $(N_{k}, B_{k})$ is an admissible pair for every $1\leq k \leq M$. Let $\mu$ be given by \eqref{def_corresponding_P_of_random_measure} with respect to $\{(N_{k}, B_{k})\}_{k=1}^M$.

		For every probability vector $\bp \in D$, let $\mathbb{P}$ be the Bernoulli measure on $\Omega$ generated by $\bp$.  By Corollary \ref{cor_fad}, $\mu$ is a spectral random measure  $\mathbb{P}$-almost surely. By  Lemma \ref{lemma_dim_expression}, we have that
		$$
		\dimh \mu=\dimp \mu= \dfrac{\sum_{k=1}^{M} p_{k} \log \# B_{k}}{\sum_{k=1}^{M} p_{k} \log N_{k}},
		$$
		for $\mathbb{P}$-a.e. $\boldsymbol{\omega} \in \Omega$.
		
		Let
		$$
		f(\bp)=\dfrac{\sum_{k=1}^{M} p_{k} \log \# B_{k}}{\sum_{k=1}^{M} p_{k} \log N_{k}}.
		$$
		Since $\max\limits_{1\leq i \leq m} \dfrac{\log \# B_{i}}{\log N_{i}} = 1$ and $\min\limits_{1\leq i \leq m} \dfrac{\log \# B_{i}}{\log N_{i}} = \frac{1}{n_{0}} < s$. By Proposition \ref{prop_dim_property}, there exists a probability vector $\mathbf{p}_0$ such that $	f(\bp_0) = s. $
		This implies that for the Bernoulli measure  $\mathbb{P}$ associated with $\mathbf{p}_0$, we have that
		\begin{equation}\label{eq_dim_a.e._equal_to_expression}
			\dimh\mu =\dimp\mu=  s
		\end{equation}
		for $\mathbb{P}$-a.e. $\boldsymbol{\omega} \in \Omega$.
	\end{proof}
	
	Notice that the dimensions  of spectral random measure is effected by the probability vector $\bp$. This reveals that random measures may have very rich geometric structures on their own.
	\begin{corollary}\label{cor_dimensions}
		Given  reals $0<a<b\leq 1$,  there exists a spectral random measure $\mu$ on $\Omega$  such  that for every $s\in (a,b)$,  we have a  Bernoulli measure  $\mathbb{P}$  on $\Omega$  such  that
$$
\dimh\mu=s,
$$
		for $\mathbb{P}$-a.e. $\boldsymbol{\omega} \in \Omega$.
	\end{corollary}
	\begin{proof}
		It is a direct consequence of the argument in Theorem \ref{thm_imp}
	\end{proof}

\end{document}